\documentclass[11pt,reqno]{amsart}

\usepackage{amsmath,amssymb,amsfonts,amsthm,hyperref,eucal,verbatim,fullpage}

\newtheorem{theorem}{Theorem}[section]
\newtheorem{lemma}[theorem]{Lemma}
\newtheorem{proposition}[theorem]{Proposition}
\newtheorem{definition}[theorem]{Definition}
\newtheorem{corollary}[theorem]{Corollary}

\newtheorem{example}[theorem]{Example}
\newtheorem{question}[theorem]{Question}
\theoremstyle{definition}
\newtheorem{remark}[theorem]{Remark}

\newenvironment{dlist}{
\begin{list}{}{\usecounter{nul}%
\setlength{\labelwidth}{6.55em}\setlength{\leftmargin}{7em}}}{\end{list}}

\newcommand{\ACst}{\hbox{$\textsf{AbC}^*$}}
\DeclareMathOperator{\bbC}{\mathbb{C}}
\DeclareMathOperator{\bbF}{\mathbb{F}}
\DeclareMathOperator{\bbN}{\mathbb{N}}

\DeclareMathOperator{\bbR}{\mathbb{R}}

\DeclareMathOperator{\A}{\mathcal{A}}
\DeclareMathOperator{\Alg}{\operatorname{Alg}}
\DeclareMathOperator{\B}{\mathcal{B}}
\DeclareMathOperator{\Ball}{\operatorname{Ball}}
\DeclareMathOperator{\C}{\mathcal{C}}
\DeclareMathOperator{\D}{\mathcal{D}}
\DeclareMathOperator{\Exp}{\operatorname{CE}}
\DeclareMathOperator{\Ext}{\operatorname{Ext}}

\renewcommand{\H}{\mathcal{H}}
\DeclareMathOperator{\I}{\mathcal{I}}
\DeclareMathOperator{\id}{\operatorname{id}}
\DeclareMathOperator{\ideal}{\lhd}
\DeclareMathOperator{\J}{\mathcal{J}}
\DeclareMathOperator{\K}{\mathcal{K}}
\renewcommand{\L}{\mathcal{L}}
\DeclareMathOperator{\Lin}{\operatorname{Lin}}
\DeclareMathOperator{\M}{\mathcal{M}}
\DeclareMathOperator{\N}{\mathcal{N}}

\let\slasho=\o 
\renewcommand{\o}{\overline}
\newcommand{\OS}{\hbox{$\textsf{OSys}_1$}}
\DeclareMathOperator{\PsExp}{\operatorname{PsExp}}
\DeclareMathOperator{\ran}{\operatorname{ran}}
\renewcommand{\S}{\mathcal{S}}
\renewcommand{\span}{\operatorname{span}}
\DeclareMathOperator{\supp}{\operatorname{supp}}
\DeclareMathOperator{\UCP}{\operatorname{UCP}}
\DeclareMathOperator{\X}{\mathcal{X}}
\DeclareMathOperator{\Z}{\mathcal{Z}}

\newcommand{\cstar}{\hbox{$C^*$}}
\newcommand{\cstaralg}{$C^*$-algebra}

\begin{document}

\title[Unique Pseudo-Expectations]{Unique Pseudo-Expectations for
  $C^*$-Inclusions} 
\author{David R. Pitts} 
\address{University of
  Nebraska--Lincoln} \email{dpitts2@math.unl.edu} 
\author{Vrej Zarikian} 
\address{U. S. Naval Academy} \email{zarikian@usna.edu}
\thanks{This work was partially supported by a grant from the Simons
  Foundation (\#316952 to David Pitts).}  
\dedicatory{Dedicated to
  E. G. Effros on the occasion of his 80th birthday}

\begin{abstract} Given an inclusion $\mathcal D \subseteq \mathcal C$
of unital $C^*$-algebras (with common unit), a unital completely positive
linear map $\Phi$ of $\mathcal C$ into the injective envelope
$I(\mathcal D)$ of $\mathcal D$ which extends the inclusion of
$\mathcal D$ into $I(\mathcal D)$ is a pseudo-expectation.
Pseudo-expectations are generalizations of conditional expectations,
but with the advantage that they always exist.  The set
$\operatorname{PsExp}(\mathcal C,\mathcal D)$ of all
pseudo-expectations is a convex set, and when $\mathcal D$ is abelian,
we prove a Krein-Milman type theorem showing that
$\operatorname{PsExp}(\mathcal C,\mathcal D)$ can be recovered from
its set of extreme points.  When $\mathcal C$ is abelian, the extreme
pseudo-expectations coincide with the homomorphisms of $\mathcal C$
into $I(\mathcal D)$ which extend the inclusion of $\mathcal D$ into
$I(\mathcal D)$, and these are in bijective correspondence with the
ideals of $\mathcal C$ which are maximal with respect to having
trivial intersection with $\mathcal D$.

  In general, $\operatorname{PsExp}(\mathcal C,\mathcal D)$ is not a
singleton.  However there are large and natural classes of inclusions
(e.g., when $\mathcal D$ is a regular MASA in $\mathcal C$) such that
there is a unique pseudo-expectation.  Uniqueness of the
pseudo-expectation typically implies interesting structural properties
for the inclusion.  For example, we show that when $\mathcal D
\subseteq \mathcal C \subseteq \mathcal B(\mathcal H)$ are von Neumann
algebras, uniqueness of the pseudo-expectation implies that $\mathcal
D' \cap \mathcal C$ is the center of $\mathcal D$; moreover, when
$\mathcal H$ is separable and $\mathcal D$ is abelian, we are able to
characterize which von Neumann algebra inclusions have the unique
pseudo-expectation property.

  For general inclusions of $C^*$-algebras with $\mathcal D$ abelian, we
give a characterization of the unique pseudo-expectation property in
terms of order structure; and when $\mathcal C$ is abelian, we are
able to give a topological description of the unique
pseudo-expectation property.

  As applications, we show that if an inclusion $\mathcal D \subseteq
\mathcal C$ has a unique pseudo-expectation $\Phi$ which is also
faithful, then the $C^*$-envelope of any operator space $\mathcal X$ with
$\mathcal D \subseteq \mathcal X \subseteq \mathcal C$ is the
$C^*$-subalgebra of $\mathcal C$ generated by $\mathcal X$; we also
show that for many interesting classes of $C^*$-inclusions, having a
faithful unique pseudo-expectation implies that $\mathcal D$ norms
$\mathcal C$, although this is not true in general.  We provide a
number of examples to illustrate the theory, and conclude with several
unresolved questions.
\end{abstract}
\maketitle



\section{Introduction}

The goal of this paper is to investigate the unique pseudo-expectation property for $C^*$-inclusions. A \emph{$C^*$-inclusion} is a pair $(\C,\D)$ of unital $C^*$-algebras with $\D \subseteq \C$ and which have
the same unit. For any unital $C^*$-algebra $\D$, there exists an \emph{injective envelope} $I(\D)$ for $\D$ \cite{Hamana1979}. That is, $I(\D)$ is an injective object in the category $\OS$ of operator systems and unital completely positive (ucp) maps, which contains $\D$, and which is minimal with respect to these two properties. In fact, $I(\D)$ is a $C^*$-algebra and $\D \subseteq I(\D)$ is a $C^*$-subalgebra. A \emph{pseudo-expectation} is a ucp map $\Phi:\C \to I(\D)$ which extends the identity map on $\D$.

Pseudo-expectations are natural generalizations of conditional expectations, and due to injectivity, have the distinct advantage that they are guaranteed to exist for any $C^*$-inclusion. Pseudo-expectations were introduced by Pitts in \cite{Pitts2012} and were used there as a replacement for conditional expectations in settings where no conditional expectation exists.

One significant difference between conditional expectations and pseudo-expectations arises when one attempts to iterate these
maps. For a conditional expectation $E:\C \to \D$, we have that $E \circ E = E$ (i.e., a conditional expectation is an idempotent
map). For a pseudo-expectation $\Phi:\C \to I(\D)$, the composition $\Phi \circ \Phi$ is typically undefined, since $I(\D)$ is usually not
contained in $\C$. This technical difficulty of pseudo-expectations is far outweighed by the aforementioned benefit, that pseudo-expectations
always exist for any $C^*$-inclusion.

We view the uniqueness and faithfulness properties of pseudo-expectations as giving a measure of the relative size of a subalgebra inside the
containing algebra. To orient the reader with this philosophy, we begin by explaining how the unique pseudo-expectation property
fits with the program of deciding when a $C^*$-subalgebra is large/substantial/rich in its containing $C^*$-algebra.

\subsection{Large Subalgebras}

Let $(\C,\D)$ be a $C^*$-inclusion. There are many ways of expressing that $\D$ is ``large'' (or ``substantial'', or ``rich'') in $\C$. For
example:
\begin{dlist}
\item[(\textsf{ARC})] The \emph{relative commutant} $\D^c = \D' \cap \C$ is abelian.
\item[(\textsf{Reg})] $\D$ is \emph{regular} in $\C$, meaning that $\o{\span}(N(\C,\D)) = \C$, where
\[
    N(\C,\D) = \{x \in \C: x\D x^* \subseteq \D, ~ x^*\D x \subseteq \D\}
\]
are the \emph{normalizers} of $\D$ in $\C$.
\item[(\textsf{Ess})] $\D$ is \emph{essential} in $\C$, meaning that every nontrivial closed two-sided ideal of $\C$ intersects $\D$ nontrivially.
\item[(\textsf{UEP})] $\C$ has the \emph{unique extension property} relative to $\D$, meaning that every pure state on $\D$ extends uniquely to a pure state on $\C$.
\item[(\textsf{Norming})] $\D$ \emph{norms} $\C$, meaning that for all $X \in M_{d \times d}(\C)$,
\[
    \|X\| = \sup\{\|RXC\|: R \in \Ball(M_{1 \times d}(\D)), ~ C \in \Ball(M_{d \times 1}(\D))\}.
\]
\end{dlist}
Some of these conditions are purely algebraic, others purely analytic, and yet others somewhere in between. Each of them has
advantages and disadvantages, and their relative merits vary by context. Indeed, two desirable properties for any condition which
``measures'' the largeness of $\D$ in $\C$ are the following:
\begin{itemize}
\item \emph{Hereditary from above}: If $\D$ is large in $\C$ and $\D \subseteq \C_0 \subseteq \C$ is a $C^*$-algebra, then $\D$ is large in $\C_0$.
\item \emph{Hereditary from below}: If $\D$ is large in $\C$ and $\D \subseteq \D_0 \subseteq \C$ is a $C^*$-algebra, then $\D_0$ is large in $\C$.
\end{itemize}
The following table shows which of these hereditary properties the various types of inclusions possess (an entry marked ``?''
indicates we do not know whether the property holds). Only conditions (\textsf{ARC}) and (\textsf{Norming}) are known to the
authors to be both hereditary from above and below:
\[
    \begin{tabular}{| c | c | c |} \hline
        Condition & Hereditary & Hereditary\\
        & from above & from below\\ \hline
        \textsf{ARC} & yes & yes \\ \hline
        \textsf{Reg} & ? & no \\ \hline
        \textsf{Ess} & no & yes \\ \hline
        \textsf{UEP} & yes & ? \\ \hline
        \textsf{Norming} & yes & yes \\ \hline
    \end{tabular}
\]
On the other hand, as the following example shows, condition (\textsf{Ess}) works the best for the particular class of abelian
inclusions, in spite of its general shortcomings.

\begin{example}
Suppose $(\A,\D) = (C(Y),C(X))$ is an abelian inclusion, with corresponding continuous surjection $j:Y \to X$. Then
\begin{itemize}
\item $(\A,\D)$ always satisfies (\textsf{ARC}).
\item $(\A,\D)$ always satisfies (\textsf{Reg}).
\item $(\A,\D)$ satisfies (\textsf{Ess}) $\iff$ the only closed set $K
  \subseteq Y$ such that $j(K) = X$ is $Y$ itself.
\item $(\A,\D)$ satisfies (\textsf{UEP}) $\iff$ $\A = \D$.
\item $(\A,\D)$ always satisfies (\textsf{Norming}).
\end{itemize}
\end{example}

\subsection{Unique Expectations}

If $\D$ is large in $\C$, then there should not be many ways to project $\C$ onto $\D$. The most natural way to project a $C^*$-algebra $\C$ onto a $C^*$-subalgebra $\D$ is via a \emph{conditional expectation}. Recall that a conditional expectation for $(\C,\D)$ is a ucp map $E:\C \to \D$ such that $E|_{\D} = \id$. A conditional expectation $E:\C \to \D$ is said to be \emph{faithful} if $E(x^*x) = 0$ implies $x = 0$ (i.e., if $E$ is faithful as a ucp map). Any convex combination of conditional expectations for $(\C,\D)$ is again a conditional expectation for
$(\C,\D)$. Thus a $C^*$-inclusion has either zero, one, or uncountably many conditional expectations, and all three possibilities can occur.

In light of the previous discussion, it is reasonable to propose the following property as yet another expression of the largeness of $\D$
in $\C$:
\begin{dlist}
\item[(\textsf{!CE})] There is at most one conditional expectation $E:\C \to \D$.
\end{dlist}
The utility of this property is seriously limited in two ways. First, for many naturally arising $C^*$-inclusions, there are no conditional
expectations at all; and second, as the next two examples show, (\textsf{!CE}) fails to be hereditary from above or below.

\begin{example} \label{!CE not hereditary from above}
Consider the $C^*$-inclusions
\[
    C[0,1] \subseteq C([0,1] \times [0,1]) \subseteq B(L^2([0,1] \times [0,1])),
\]
where the first inclusion corresponds to the continuous surjection
\[
    j:[0,1] \times [0,1] \to [0,1]:(s,t) \mapsto s.
\]
Then there are no conditional expectations for the inclusion
$(B(L^2([0,1] \times [0,1])),C[0,1])$, since $C[0,1]$ is not injective
(in the category $\OS$). But there are infinitely many conditional expectations for the
inclusion $(C([0,1] \times [0,1]),C[0,1])$. Indeed,
\[
    E_t:C([0,1] \times [0,1]) \to C[0,1]: g \mapsto g(\cdot,t)
\]
is a conditional expectation for each $t \in [0,1]$. Thus
(\textsf{!CE}) is not hereditary from above.
\end{example}

\begin{example} \label{!CE not hereditary from below}
Likewise, consider the $C^*$-inclusions
\[
    C[0,1] \subseteq L^\infty[0,1] \subseteq B(L^2[0,1]).
\]
Then there are no conditional expectations for the inclusion
$(B(L^2[0,1]),C[0,1])$, but infinitely many conditional expectations
for the inclusion $(B(L^2[0,1]),L^\infty[0,1])$
\cite{KadisonSinger1959}. Thus (\textsf{!CE}) is not hereditary from
below.
\end{example}

\subsection{Unique Pseudo-Expectations}

Recall that a ucp map $\Phi:\C\rightarrow I(\D)$ is a
pseudo-expectation if it extends the inclusion of $\D$ into
$I(\D)$. Clearly every conditional expectation for $(\C,\D)$ is a
pseudo-expectation for $(\C,\D)$, so pseudo-expectations generalize
conditional expectations. But pseudo-expectations always exist for any
$C^*$-inclusion.

With the discussion of the previous section in mind, we are led to replace condition (\textsf{!CE}) there by the following stronger condition:
\begin{dlist}
\item[(\textsf{!PsE})] There exists a unique pseudo-expectation for $(\C,\D)$.
\end{dlist}
Or perhaps by the even stronger condition:
\begin{dlist}
\item[(\textsf{f!PsE})] There exists a unique pseudo-expectation for $(\C,\D)$, which is faithful.
\end{dlist}
We will see shortly that both of these conditions are hereditary from above (Proposition \ref{hereditary from above}). Compelling evidence
that (\textsf{!PsE}) and (\textsf{f!PsE}) are closely related to the largeness of $\D$ in $\C$ is provided by a striking result from
\cite{Pitts2012}:

\begin{theorem}[Pitts] \label{Pitts} Let $(\C,\D)$ be a regular
inclusion with $\D$ a MASA in $\C$.
\begin{enumerate}
\item Then there exists a unique pseudo-expectation $\Phi:\C \to I(\D)$.
\item If $\L_\Phi = \{x \in \C: \Phi(x^*x) = 0\}$ is the left kernel of $\Phi$, then $\L_\Phi$ is the unique maximal $\D$-disjoint
ideal in $\C$.
\item If $\Phi$ is faithful (i.e., if $\L_\Phi = 0$), then $\D$ norms $\C$.
\end{enumerate}
\end{theorem}

Rephrasing Theorem \ref{Pitts} using the notation of this section, statement (i) says that for a $C^*$-inclusion $(\C,\D)$, with $\D$
maximal abelian,
\[
    \text{(\textsf{Reg})} \implies \text{(\textsf{!PsE})}.
\]
Statements (ii) and (iii) imply that under the same hypotheses,
\[
    \text{(\textsf{Reg})} \wedge \text{(\textsf{f!PsE})} \implies \text{(\textsf{Ess})} \wedge \text{(\textsf{Norming})}.
\]

This paper is  a systematic attempt to generalize Theorem \ref{Pitts}. We characterize the unique pseudo-expectation property for various important classes of $C^*$-inclusions, and we relate the unique pseudo-expectation property
for a $C^*$-inclusion $(\C,\D)$ to other measures of the largeness of $\D$ in $\C$, in particular conditions (\textsf{ARC}), (\textsf{Reg}),
(\textsf{Ess}), (\textsf{UEP}), and (\textsf{Norming}) above. Necessarily, we significantly develop the general theory of
pseudo-expectations along the way.

\section{The Unique Pseudo-Expectation Property}

\subsection{Definitions and Basic Properties} \label{basics section}

In this section we formally define pseudo-expectations and explore
their basic properties. Before doing so, we remind the reader of a few
facts about injective envelopes and establish some standing assumptions used
throughout the paper. All \cstar-algebras are assumed unital, and
homomorphisms between \cstaralg s will always be $*$-homomorphisms
which preserve the units. We will denote by $\OS$ the category whose
objects are operator systems and whose morphisms are ucp (unital
completely positive) maps. A \cstaralg\ is \textit{injective} if it is
injective when viewed as an object in \OS. Let $\ACst$ be the category
of abelian \cstar-algebras and homomorphisms. Clearly every object in
$\ACst$ is also an object in $\OS$. An important observation found in
\cite{Hamana1979a} and \cite{HadwinPaulsen2011} is that an abelian
\cstaralg\ is injective in $\ACst$ if and only if it is injective in
\OS.

\begin{theorem}[see \cite{EffrosRuan2000} or \cite{Paulsen2002}] \label{injEnv}
Let $\D$ be a unital \cstaralg. Then there exists a unital \cstaralg\ $\A$ and a unital $*$-monomorphism $\iota:\D \rightarrow \A$ with the following properties:
\begin{enumerate}
\item $\A$ is injective;
\item if $\S$ is an injective object in $\OS$ and $\tau:\D \rightarrow \S$ is a unital complete isometry, then there exists a unital
  complete isometry $\tau_1:\A \rightarrow \S$ such that $\tau = \tau_1 \circ \iota$.
\end{enumerate}
\end{theorem}

The pair $(\A,\iota)$ is called an \textit{injective envelope} for $\D$, and it is ``nearly'' unique. The ambiguity arises from the fact that in general, the choice of $\tau_1$ in Theorem \ref{injEnv} is not unique. However, in the sequel, we will assume that for a given \cstaralg\ $\D$ under discussion, a choice of injective envelope $(I(\D),\iota)$  has been made. Furthermore, we will regard $\iota$ as an inclusion map and suppress writing it. Thus we will always regard $\D$ as a $C^*$-subalgebra of $I(\D)$.

\begin{definition}
  A \textbf{pseudo-expectation} for the $C^*$-inclusion $(\C,\D)$ is a
  ucp map $\Phi:\C \to I(\D)$ such that $\Phi|_{\D} = \id$. We denote
  by $\PsExp(\C,\D)$ the collection of all pseudo-expectations for
  $(\C,\D)$.
\end{definition}

\begin{proposition} \label{basic properties Phi} Let $(\C,\D)$ be a
  $C^*$-inclusion, $\Phi \in \PsExp(\C,\D)$, and  \[\L_\Phi = \{x \in
  \C: \Phi(x^*x) = 0\}\] be the left kernel of $\Phi$. Then the
  following statements hold:
\begin{enumerate}
\item $\Phi$ is a $\D$-bimodule map. That is, $\Phi(d_1xd_2) =
  d_1\Phi(x)d_2$ for all $x \in \C$, $d_1, d_2 \in \D$.
\item $\L_\Phi$ is a closed left ideal in $\C$ which intersects $\D$
  trivially. Furthermore, $\L_\Phi$ is a right $\D$-module.
\end{enumerate}
\end{proposition}

\begin{proof}
The first statement follows from Choi's Lemma (cf.~\cite[Corollary~3.19]{Paulsen2002}); the second is straightforward.
\end{proof}

\begin{proposition} \label{basic properties PsExp} Let $(\C,\D)$ be a
$C^*$-inclusion.
\begin{enumerate}
\item The collection $\PsExp(\C,\D)$ of all pseudo-expectations
for $(\C,\D)$ forms a nonempty convex subset of $\UCP(\C,I(\D))$, the
ucp maps from $\C$ into $I(\D)$. In fact, $\PsExp(\C,\D)$ is a face of
$\UCP(\C,I(\D))$. Thus any extreme point of $\PsExp(\C,\D)$ is an
extreme point of $\UCP(\C,I(\D))$.
\item If $\Exp(\C,\D)$ denotes the collection of all conditional
expectations for $(\C,\D)$, then $\Exp(\C,\D) \subseteq
\PsExp(\C,\D)$. Of course it can happen that $\Exp(\C,\D) =
\emptyset$, whereas $\PsExp(\C,\D) \neq \emptyset$, by injectivity.
\end{enumerate}
\end{proposition}

\begin{proof} We only prove that $\PsExp(\C,\D)$ is a face of
$\UCP(\C,I(\D))$. Indeed, suppose $\Phi \in \PsExp(\C,\D)$ and $\Phi =
\lambda\Phi_1 + (1-\lambda)\Phi_2$, where $\Phi_1, \Phi_2 \in
\UCP(\C,I(\D))$ and $\lambda \in (0,1)$. For any $u \in U(\D)$ (the
unitary group of $\D$), we have that
\[ u = \Phi(u) = \lambda\Phi_1(u) + (1-\lambda)\Phi_2(u).
\] Since $\Phi_1(u), \Phi_2(u) \in \Ball(I(\D))$ and $u \in U(\D)
\subseteq U(I(\D)) \subseteq \Ext(\Ball(I(\D))$, we conclude that
$\Phi_1(u) = \Phi_2(u) = u$. It follows that $\Phi_1(d) = \Phi_2(d) =
d$ for all $d \in \D$, so that $\Phi_1, \Phi_2 \in \PsExp(\C,\D)$.
\end{proof}

\begin{definition} We say that a $C^*$-inclusion $(\C,\D)$ has the
\textbf{unique pseudo-expectation property} (\textsf{!PsE}) if there
exists a unique $\Phi \in \PsExp(\C,\D)$. If, in addition, $\Phi$ is
faithful, then we say that $(\C,\D)$ has the \textbf{faithful unique
pseudo-expectation property} (\textsf{f!PsE}).
\end{definition}

As in the introduction, we say that a property of $C^*$-inclusions is
\emph{hereditary from above} if whenever $(\C,\D)$ has the property
and $\D \subseteq \C_0 \subseteq \C$ is a $C^*$-algebra, then
$(\C_0,\D)$ has the property.

\begin{proposition} \label{hereditary from above} The unique
pseudo-expectation property is hereditary from above, as is the
faithful unique pseudo-expectation property.
\end{proposition}

\begin{proof} Suppose $\PsExp(\C,\D) = \{\Phi\}$. Let $\D \subseteq
\C_0 \subseteq \C$ be a $C^*$-algebra, and fix $\theta \in
\PsExp(\C_0,\D)$. By injectivity, there exists a ucp map $\Theta:\C
\to I(\D)$ such that $\Theta|_{\C_0} = \theta$. Then $\Theta|_{\D} =
\theta|_{\D} = \id$, so that $\Theta \in \PsExp(\C,\D)$. It follows
that $\Theta = \Phi$, which implies $\theta = \Theta|_{\C_0} =
\Phi|_{\C_0}$. Thus $\PsExp(\C_0,\D) = \{\Phi|_{\C_0}\}$. If $\Phi$ is
faithful, then so is $\Phi|_{\C_0}$.
\end{proof}

On the other hand, if $(\C,\D)$ has the unique pseudo-expectation
property and $\D \subseteq \D_0 \subseteq \C$ is a $C^*$-algebra, then
$(\C,\D_0)$ may not have the unique pseudo-expectation property (see
Example \ref{not hereditary from below}). That is, the unique
pseudo-expectation property is \underline{not} \emph{hereditary from
below}.

\subsection{Elementary Examples}

In this section we give some examples of $C^*$-inclusions with (and
without) the unique pseudo-expectation property. These examples are
``elementary'', insofar as we can prove that they are actually
examples without any additional technology. Later, after we have
developed some general theory for pseudo-expectations, we will give a
number of ``advanced'' examples.

\begin{example}[regular MASA inclusions] Let $(\C,\D)$ be a regular
MASA inclusion. Then $(\C,\D)$ has the unique pseudo-expectation
property, by Pitts' Theorem \ref{Pitts}. Two classes of regular MASA
inclusions which appear in the literature are \textbf{$C^*$-diagonals}
in the sense of Kumjian \cite{Kumjian1986}, and \textbf{Cartan
subalgebras} in the sense of Renault \cite{Renault2008}.

\end{example} \begin{example}[atomic MASA] \label{atomic MASA} The
inclusion $(B(\ell^2),\ell^\infty)$ has the faithful unique
pseudo-expectation property. Indeed, $\ell^\infty$ is injective (since
it is an abelian $W^*$-algebra) and there exists a unique conditional
expectation $E:B(\ell^2) \to \ell^\infty$, which is faithful
\cite[Theorem 1]{KadisonSinger1959}.
\end{example}

\begin{example}[diffuse MASA] \label{diffuse MASA} The inclusion
$(B(L^2[0,1]),L^\infty[0,1])$ has infinitely many pseudo-expectations,
none of which are faithful. On the other hand, the inclusion $(L^\infty[0,1]+K(L^2[0,1]),L^\infty[0,1])$ has a unique pseudo-expectation, which is not faithful.
\end{example}

\begin{proof}
Since $L^\infty[0,1]$ is injective, conditional expectations and pseudo-expectations for $(B(L^2[0,1]),L^\infty[0,1])$ are the same. By Theorem 2 and Remark 5 of \cite{KadisonSinger1959}, there are infinitely many conditional expectations $B(L^2[0,1]) \to L^\infty[0,1]$, all of which annihilate $K(L^2[0,1])$. Now suppose $E:L^\infty[0,1]+K(L^2[0,1]) \to L^\infty[0,1]$ is a conditional expectation. Then $E$ extends to a conditional expectation $\tilde{E}:B(L^2[0,1]) \to L^\infty[0,1]$. Thus, by the previous discussion,
\[
    E(d + h) = \tilde{E}(d + h) = d
\]
for all $d \in L^\infty[0,1]$, $h \in K(L^2[0,1])$.
\end{proof}

\begin{remark}
Let $(\C,\D)$ be a $C^*$-inclusion. Then we have $C^*$-inclusions $\D \subseteq \C \subseteq I(\C)$. By Theorem \ref{injEnv}, it follows that we have an operator system inclusion $I(\D) \subseteq I(\C)$. If $\D$ is abelian, then in fact we have a $C^*$-inclusion $I(\D) \subseteq I(\C)$ \cite[Thm. 2.21]{HadwinPaulsen2011}. In that case, if $\Phi:\C \to I(\D)$ is a pseudo-expectation for $(\C,\D)$, then it is not hard to see that any ucp extension $\tilde{\Phi}:I(\C) \to I(\D)$ of $\Phi$ is a conditional expectation for $(I(\C),I(\D))$. As the previous example shows, this extension need not be unique. Indeed, by \cite[Ex. 5.3]{Hamana1979}, $I(L^\infty[0,1]+K(L^2[0,1])) = B(L^2[0,1])$.
\end{remark}

Next we consider $C^*$-inclusions $(\C,\D)$ such that there is a monomorphism of $\C$ into $I(\D)$.  By \cite[Lemma 4.6]{Hamana1979}, these are precisely the operator space essential inclusions. A $C^*$-inclusion $(\C,\D)$ is \emph{operator space essential} (\textsf{OSE}) if every complete contraction $u:\C \to B(\H)$ which is completely isometric on $\D$ is actually completely isometric on $\C$.

\begin{example}[\textsf{OSE} inclusions] \label{OSE} Let $\D$ be an
  arbitrary unital $C^*$-algebra. Then $(I(\D),\D)$ has the faithful
  unique pseudo-expectation property. More generally, if $\D \subseteq
  \C \subseteq I(\D)$ are $C^*$-inclusions, then $(\C,\D)$ has the
  faithful unique pseudo-expectation property, by Proposition
  \ref{hereditary from above}.
\end{example}

\begin{proof}
  Let $\Phi \in \PsExp(I(\D),\D)$. Then $\Phi:I(\D) \to I(\D)$ is a
  ucp map such that $\Phi|_{\D} = \id$. By the \emph{rigidity} of the
  injective envelope, $\Phi = \id$.
\end{proof}

\begin{example}[\textsf{UEP} MASA inclusions] \label{UEP MASA} Let
$(\C,\D)$ be a $C^*$-inclusion, with $\D$ abelian. Assume that
$(\C,\D)$ has the \textbf{unique extension property} (\textsf{UEP}),
meaning that every pure state on $\D$ extends uniquely to a pure state
on $\C$. (This forces $\D$ to be a MASA in $\C$.) Then $(\C,\D)$ has
the unique pseudo-expectation property. In fact, the unique
pseudo-expectation is a conditional expectation.
\end{example}

\begin{proof} By \cite[Cor. 2.7]{ArchboldBunceGregson1982}, we have
the direct sum decomposition
\[ \C = \D + \o{\span}\{[\C,\D]\}.
\] If $\Phi \in \PsExp(\C,\D)$, then by Proposition \ref{basic
properties Phi} and the fact that $I(\D)$ is abelian,
\[ \Phi(xd - dx) = \Phi(xd) - \Phi(dx) = \Phi(x)d - d\Phi(x) = 0, ~ x
\in \C, ~ d \in \D.
\] The result follows.
\end{proof}

\begin{remark} Initially the study of \textsf{UEP} inclusions
  $(\C,\D)$ focused on the case $\D$ abelian, and there has been
  substantial work in this direction.  Later work has made progress in
  the general setting \cite{BunceChu1998}. It would be interesting to
  know whether a general \textsf{UEP} inclusion $(\C,\D)$ has the
  unique pseudo-expectation property.  A possible test case to
  consider is the inclusion $(C_r^*(\bbF_m),C_r^*(\bbF_n))$, for
  $m > n \geq 2$ \cite[Thm. 2.6]{AkemannWassermannWeaver2010}.
\end{remark}

\begin{example} Let $\M$ be a $II_1$ factor with separable predual and
$\D \subseteq \M$ be a MASA. More generally, let $\M$ be any $II_1$
factor and $\D \subseteq \M$ be a singly-generated MASA. Then
$(\M,\D)$ \underline{does not} have the unique pseudo-expectation property
\cite[Thm. 4.4]{AkemannSherman2012}.
\end{example}

\section{Some General Theory}

In this section we prove some general results about
pseudo-expectations, which we will use later to analyze more
complicated examples than those considered so far.

\subsection{Left Kernel}

Let $(\C,\D)$ be a $C^*$-inclusion. We say that a closed two-sided
ideal $\J \ideal \C$ is \emph{$\D$-disjoint} if $\D \cap \J = 0$. It
is not hard to prove that every $\D$-disjoint ideal of $\C$ is
contained in a maximal $\D$-disjoint ideal of $\C$.

As seen in Theorem \ref{Pitts}, if $(\C,\D)$ is a regular MASA
inclusion, then there exists a unique maximal $\D$-disjoint ideal in
$\C$, namely the left kernel $\L_\Phi$ of the unique
pseudo-expectation $\Phi:\C\rightarrow I(\D)$.  In  general, for a
$C^*$-inclusion $(\C,\D)$ with unique pseudo-expectation $\Phi$, the
left kernel $\L_\Phi$ is only a left ideal of $\C$, rather than a
two-sided ideal (see Example \ref{Calkin} below).   Nevertheless, we have the
following structural result for general \cstar-inclusions with the
unique pseudo-expectation property.

\begin{proposition} \label{left kernel} Let $(\C,\D)$ be a
$C^*$-inclusion. If $(\C,\D)$ has unique pseudo-expectation $\Phi$,
then there exists a unique maximal $\D$-disjoint ideal $\I \ideal
\C$. Furthermore, $\I \subseteq \L_\Phi$, the left kernel of $\Phi$.
\end{proposition}

\begin{proof} Let $\J \ideal \C$ be a $\D$-disjoint ideal. Then the
map $\D + \J \to \D: d + h \mapsto d$ is a unital $*$-homomorphism,
which extends by injectivity to a pseudo-expectation for $(\C,\D)$,
necessarily $\Phi$. Thus $\J \subseteq \ker(\Phi)$. If $h \in \J$,
then $h^*h \in \J$, which implies $\Phi(h^*h) = 0$, which in turn
implies $h \in \L_\Phi$. Thus $\J \subseteq \L_\Phi$. It follows that
\[ \cup\{\J: \J \ideal \C, ~ \D \cap \J = 0\} \subseteq \L_\Phi,
\] and so
\[ \I = \o{\span}(\cup\{\J: \J \ideal \C, ~ \D \cap \J = 0\})
\subseteq \L_\Phi.
\] Thus $\I$ is the unique maximal $\D$-disjoint ideal of $\C$.
\end{proof}

\subsection{Characterization: Every Pseudo-Expectation is Faithful}

In this section we characterize the property ``every
pseudo-expectation is faithful'' for arbitrary $C^*$-inclusions
$(\C,\D)$ in terms of the (hereditary) $\D$-disjoint ideal structure
of $\C$ (Theorem \ref{faithful characterization}). Formally, the
property ``every pseudo-expectation is faithful'' is weaker than the
faithful unique pseudo-expectation property. On the other hand, we
have no examples showing that it is strictly weaker. So in principle,
Theorem \ref{faithful characterization} could be a characterization of
the faithful unique pseudo-expectation property. We list this as an
open problem.

\begin{question} \label{faithful unique question} Does the property
``every pseudo-expectation is faithful'' imply the faithful unique
pseudo-expectation property?
\end{question}

To proceed with our characterization, we will need two notions from
earlier in the paper. First, recall that a closed two-sided ideal $\J
\ideal \C$ is \emph{$\D$-disjoint} if $\D \cap \J = 0$. Second, recall
that a $C^*$-inclusion $(\C,\D)$ is \emph{essential} (\textsf{Ess}) if
every nontrivial closed two-sided ideal of $\C$ intersects $\D$
nontrivially. The following proposition relates these two notions with
each other, as well as to a useful mapping property.

\begin{proposition}
Let $(\C,\D)$ be a $C^*$-inclusion. Then the following are equivalent:
\begin{enumerate}
\item $(\C,\D)$ is essential.
\item The only $\D$-disjoint ideal of $\C$ is $0$.
\item Whenever $\pi:\C \to B(\H)$ is a unital $*$-homomorphism such that $\pi|_{\D}$ is faithful, then $\pi$ itself is faithful.
\end{enumerate}
\end{proposition}

\begin{proof}
(i $\iff$ ii) Tautological.

(ii $\implies$ iii) Suppose the only $\D$-disjoint ideal of $\C$ is the
trivial ideal. Let $\pi:\C \to
B(\H)$ be a unital $*$-homomorphism such that $\pi|_{\D}$ is
faithful. Then $\ker(\pi) \ideal \C$ is a $\D$-disjoint ideal. By
assumption, $\ker(\pi) = 0$, so $\pi$ is faithful.

(iii $\implies$ ii) Conversely, suppose that for every unital
$*$-homomorphism $\pi:\C \to B(\H)$, $\pi$ is faithful whenever
$\pi|_{\D}$ is faithful. Let $\J \ideal \C$ be a $\D$-disjoint
ideal. Then $q:\C \to \C/\J: x \mapsto x + \J$ is a unital
$*$-homomorphism such that $q|_{\D}$ is faithful. By assumption, $q$
is faithful, so $\J = 0$.
\end{proof}

As we saw in the introduction, the condition (\textsf{Ess}) is not
hereditary from above. Indeed, $(M_{2 \times 2}(\bbC),\bbC I)$
satisfies (\textsf{Ess}), since $M_{2 \times 2}(\bbC)$ is simple, and
$\bbC I \subseteq \bbC \oplus \bbC \subseteq M_{2 \times 2}(\bbC)$ is
a $C^*$-algebra, but $(\bbC \oplus \bbC,\bbC I)$ fails
(\textsf{Ess}). To resolve this issue, we introduce the following
stronger condition:

\begin{definition} We say that a $C^*$-inclusion $(\C,\D)$ is
\textbf{hereditarily essential} if $(\C_0,\D)$ is essential whenever
$\D \subseteq \C_0 \subseteq \C$ is a $C^*$-algebra.
\end{definition}

Now comes the promised characterization.

\begin{theorem} \label{faithful characterization}
Let $(\C,\D)$ be a $C^*$-inclusion. Then the following are equivalent:
\begin{enumerate}
\item Every pseudo-expectation $\Phi \in \PsExp(\C,\D)$ is faithful.
\item $(\C,\D)$ is hereditarily essential.
\end{enumerate}
\end{theorem}

\begin{proof} (i $\implies$ ii) Suppose that every pseudo-expectation
$\Phi \in \PsExp(\C,\D)$ is faithful. Let $\D \subseteq \C_0 \subseteq
\C$ be a $C^*$-algebra and $\J_0 \ideal \C_0$ be a $\D$-disjoint
ideal. Then $\Phi_0:\D + \J_0 \to \D: d + h \mapsto d$ is a unital
$*$-homomorphism. By injectivity, there exists $\Phi \in
\PsExp(\C,\D)$ such that such that $\Phi|_{\D + \J_0} = \Phi_0$. Since
$\Phi$ is faithful, so is $\Phi_0$, which implies $\J_0 = 0$. It
follows that $(\C_0,\D)$ is essential, which implies $(\C,\D)$ is
hereditarily essential.

(ii $\implies$ i) Conversely, suppose that $(\C,\D)$ is hereditarily
essential. Let $\Phi \in \PsExp(\C,\D)$ and $x \in \L_\Phi$ (the left
kernel of $\Phi$). Define $\C_0 = C^*(\D,|x|)$, so that $\D \subseteq
\C_0 \subseteq \C$, and let $\J_0 \ideal \C_0$ be the closed two-sided
ideal generated by $|x|$. We claim that $\J_0 \subseteq
\L_\Phi$. Indeed, $\J_0 = \overline{\span}\{w|x|d: w \in \C_0, ~ d \in
\D\}$ and $\L_\Phi$ is both a closed left ideal and a right
$\D$-module in $\C$ containing $|x|$ (Proposition \ref{basic
properties Phi}). Since $\D \cap \L_\Phi = 0$, $\D \cap \J_0 = 0$, and
since $(\C_0,\D)$ is essential by assumption, $\J_0 = 0$. Thus $|x| =
0$, which implies $x = 0$. Hence $\L_\Phi = 0$, so $\Phi$ is faithful.
\end{proof}

\subsection{Quotients} \label{quotient section}

We next examine the behavior of the unique
pseudo-expectation property with respect to quotients. Specifically,
we are interested to know when the unique pseudo-expectation property
for $(\C,\D)$ passes to $(\C/\J,\D/(\J \cap \D))$, for a closed
two-sided ideal $\J \ideal \C$. If $\J \cap \D = 0$, then the answer
is ``always'', and faithfulness is preserved.

\begin{proposition} Let $(\C,\D)$ be a $C^*$-inclusion and $\J \ideal
\C$ be a $\D$-disjoint ideal. If $(\C,\D)$ has the unique
pseudo-expectation property, then so does $(\C/\J,\D)$. If $(\C,\D)$
has the faithful unique pseudo-expectation property, then so does
$(\C/\J,\D)$ (trivially, because $\J = 0$).
\end{proposition}

\begin{proof} Suppose $\PsExp(\C,\D) = \{\Phi\}$. Let $\theta \in
\PsExp(\C/\J,\D)$. Then $\theta \circ q \in \PsExp(\C,\D)$, where
$q:\C \to \C/\J$ is the quotient map. Thus $\theta \circ q = \Phi$,
which implies $\theta(x + \J) = \Phi(x)$, $x \in \C$. Hence
$\PsExp(\C/\J,\D) = \{\theta\}$. If $\Phi$ is faithful, then $\J = 0$,
by Theorem \ref{faithful characterization}.
\end{proof}

\begin{remark} If $\D \cap \J \neq 0$, then it is entirely possible
that $(\C,\D)$ has a unique pseudo-expectation but $(\C/\J,\D/(\J \cap
\D))$ does not (see Example \ref{non-unique quotient}).
\end{remark}

In order to obtain a positive result when $\D \cap \J \neq 0$, we
require $\J \cap \D \ideal \D$ to be regular. Recall that if $\A$ is a
unital $C^*$-algebra and $\I \ideal \A$, then \[\I^\perp = \{a \in \A:
a\I = \I a = 0\} \ideal \A.\] Also, $\I$ is \emph{regular} if
$\I^{\perp\perp} = (\I^\perp)^\perp = \I$.
Combining~\cite[Lemma~1.3(iii)]{Hamana1982}
with~\cite[Theorem~6.3]{HamanaReEmCStAlMoCoCStAl}, one finds that
given a regular ideal $\I \ideal \A$, there exists a unique projection
$p \in Z(I(\A))$ such that $\I = \{a \in \A: ap = a\}$.
In that case, the unital $*$-isomorphism $\A/\I \to
\A p^\perp: a + \I \mapsto ap^\perp$ extends uniquely to a unital
$*$-isomorphism $I(\A/\I) \cong I(\A)p^\perp$.

\begin{theorem} \label{quotient} Let $(\C,\D)$ be a $C^*$-inclusion
and $\J \ideal \C$. If $(\C,\D)$ has the unique pseudo-expectation
property and $\J \cap \D \ideal \D$ is regular, then $(\C/\J,\D/(\J
\cap \D))$ has the unique pseudo-expectation property.
\end{theorem}

\begin{proof} Let $p \in Z(I(\D))$ be the unique projection such that
$\J \cap \D = \{d \in \D: dp = d\}$. Then the unital $*$-isomorphism
$\D/(\J \cap \D) \to \D p^\perp: d + (\J \cap \D) \mapsto dp^\perp$
extends uniquely to a unital $*$-isomorphism $I(\D/(\J \cap \D)) \cong
I(\D)p^\perp$. Now suppose $\PsExp(\C,\D) = \{\Phi\}$ and let $\theta
\in \PsExp(\C/\J,\D/(\J \cap \D))$. Then $\theta:\C/\J \to
I(\D)p^\perp$ is a ucp map such that $\theta(d + \J) = dp^\perp$, $d
\in \D$. Define $\Theta:\C \to I(\D)$ by
\[ \Theta(x) = \theta(x + \J) + \Phi(x)p, ~ x \in \C.
\] Then $\Theta \in \PsExp(\C,\D)$, which implies $\Theta = \Phi$,
which in turn implies
\[ \theta(x + \J) = \Phi(x)p^\perp, ~ x \in \C.
\] Thus $(\C/\J,\D/(\J \cap \D))$ has the unique pseudo-expectation
property.
\end{proof}

\begin{remark} If $(\C,\D)$ has the faithful unique pseudo-expectation
property and $\J \ideal \C$, then $(\C/\J,\D/(\J \cap \D))$ need not
have the faithful unique pseudo-expectation property, even if $\J \cap
\D \ideal \D$ is regular (see Example \ref{unfaithful quotient}).
\end{remark}

A very interesting example not covered by the results of this section
occurs when $\C = B(\ell^2)$, $\D = \ell^\infty$, and $\J =
K(\ell^2)$, so that
\[ (\C/\J,\D/(\J \cap \D)) = (B(\ell^2)/K(\ell^2),\ell^\infty/c_0).
\] Indeed, $c_0 \ideal \ell^\infty$ is not regular, since
$c_0^{\perp\perp} = \ell^\infty$. Our analysis of this example is
greatly simplified by the recent remarkable affirmative solution to
the Kadison-Singer Problem \cite{MarcusSpielmanSrivastava2015}.

\begin{example}[Calkin algebra] \label{Calkin} The inclusion
$(B(\ell^2)/K(\ell^2),\ell^\infty/c_0)$ has the unique
pseudo-expectation property. In fact, the unique pseudo-expectation is
a conditional expectation which is \underline{not} faithful.
\end{example}

\begin{proof} By \cite{MarcusSpielmanSrivastava2015}, the inclusion
$(B(\ell^2),\ell^\infty)$ has the unique extension property
(\textsf{UEP}). By \cite[Lemma 3.1]{ArchboldBunceGregson1982},
$(B(\ell^2)/K(\ell^2),\ell^\infty/c_0)$ has (\textsf{UEP}) as
well. Thus $(B(\ell^2)/K(\ell^2),\ell^\infty/c_0)$ has a unique
pseudo-expectation $\tilde{E}$, which is actually a conditional
expectation, by Example \ref{UEP MASA}. In fact,
\[ \tilde{E}(x + K(\ell^2)) = E(x) + c_0, ~ x \in B(\ell^2),
\] where $E:B(\ell^2) \to \ell^\infty$ is the unique conditional
expectation. Letting $h \in B(\ell^2)_+$ be the \emph{Hilbert matrix}
\cite{Choi1983}, we see that $\tilde{E}(h + K(\ell^2)) = 0$, but $h +
K(\ell^2) \neq 0$.
\end{proof}

\begin{remark} \label{not ideal} Example \ref{Calkin} furnishes an
instance of a $C^*$-inclusion $(\C,\D)$ with a unique
pseudo-expectation $\Phi$, such that $\L_\Phi$ is not a two-sided
ideal of $\C$. Indeed, $\C$ is simple but $\L_\Phi \neq 0$ in Example
\ref{Calkin}. This should be compared with Theorem \ref{Pitts}.
\end{remark}

\subsection{Abelian Relative Commutant}

As mentioned in the introduction, the unique pseudo-expectation
property for a $C^*$-inclusion $(\C,\D)$ can be thought of as an
expression of the fact that $\D$ is ``large'' in $\C$. A more familiar
algebraic expression of the largeness of $\D$ in $\C$ is that $\D^c =
\D' \cap \C$, the \emph{relative commutant} of $\D$ in $\C$, is
``small'' (abelian). In Corollary \ref{!fPse implies ARC} below, we
show that the \underline{faithful} unique pseudo-expectation property
implies that the relative commutant is abelian, symbolically
\[ \text{(\textsf{f!PsE})} \implies \text{(\textsf{ARC})}.
\] We expect that the hypothesis of faithfulness is not needed for
this result, but we have not been able to eliminate it.

\begin{theorem} \label{not ARC} Let $(\C,\D)$ be a
$C^*$-inclusion. Assume that there exists a faithful
pseudo-expectation $\Phi \in \PsExp(\C,\D)$. If $\D^c$ is not abelian,
then there exist infinitely many pseudo-expectations for $(\C,\D)$,
some of which are not faithful.
\end{theorem}

\begin{proof} We may assume that $\C \subseteq B(\H)$ for some Hilbert
space $\H$. If $\D^c$ is not abelian, then
there exists $x \in \D^c$
with $\|x\| = 1$ and $x^2 = 0$ (\cite[p.~288]{DoranBelfiChC*Al}). Let $x = u|x|$ be the polar
decomposition, so that $u \in \D'$ is a partial isometry with initial
space $\o{\ran}(|x|)$ and final space $\o{\ran}(x)$. Since
\[ \o{\ran}(x) \subseteq \ker(x) = \ker(|x|) = \ran(|x|)^\perp,
\] we find that $u^2 = 0$. It follows that
\[ p_1 = u^*u, ~ p_2 = uu^*, \text{ and } p_3 = 1 - u^*u - uu^*
\] are orthogonal projections in $\D'$. For $\lambda \in [0,1]$ define
$\theta_\lambda:B(\H) \to B(\H)$ by
\[ \theta_\lambda(t) = \lambda p_1tp_1 + (1-\lambda)u^*tu + \lambda
utu^* + (1-\lambda) p_2tp_2 + p_3tp_3.
\] Then $\theta_\lambda$ is a ucp map such that
\[ \theta_\lambda|_{\D} = \id, ~ \theta_\lambda(x^*x) = \lambda(x^*x +
xx^*), \text{ and } \theta_\lambda(xx^*) = (1-\lambda)(x^*x + xx^*).
\] Consider the operator system
\[ \mathcal{S} := \D + \bbC x^*x + \bbC xx^*\subseteq \C.
\]  Since $\theta_\lambda(\mathcal{S})
\subseteq \mathcal{S}$,
\[ \Phi_\lambda^0 := \Phi \circ
\theta_\lambda|_{\mathcal{S}}:\mathcal{S} \to I(\D)
\] is a well-defined ucp map such that
\[ \Phi_\lambda^0|_{\D} = \id, ~ \Phi_\lambda^0(x^*x) =
\lambda\Phi(x^*x + xx^*), \text{ and } \Phi_\lambda^0(xx^*) =
(1-\lambda)\Phi(x^*x + xx^*).
\] By injectivity, there exists $\Phi_\lambda \in \PsExp(\C,\D)$ such
that $\Phi_\lambda|_{\mathcal{S}} = \Phi_\lambda^0$. Since $\Phi$ is
faithful, $\Phi(x^*x + xx^*) \neq 0$, and so $\Phi_\lambda \neq
\Phi_\mu$ if $\lambda \neq \mu$. Consequently $\{\Phi_\lambda: \lambda
\in [0,1]\}$ is an infinite family of pseudo-expectations for
$(\C,\D)$, some of which are not faithful (namely $\Phi_0$ and
$\Phi_1$).
\end{proof}

\begin{remark} \label{halving} In Theorem \ref{not ARC}, we may remove
the hypothesis that there exists a faithful pseudo-expectation $\Phi
\in \PsExp(\C,\D)$, provided we strengthen the hypothesis on
$\D^c$. For example, we could ask that $\D^c$ contain a halving
projection. In that case the proof simplifies substantially.
\end{remark}

\begin{corollary} \label{!fPse implies ARC} Let $(\C,\D)$ be a
$C^*$-inclusion. If $(\C,\D)$ has the faithful unique
pseudo-expectation property, then $\D^c$ is abelian.
\end{corollary}

\begin{remark} \label{ARC remark} As indicated earlier, we expect
Corollary \ref{!fPse implies ARC} to remain true without the
assumption of faithfulness. At this point, however, we do not have a
proof, even in the case $\D$ abelian. On the other hand, the case of
$W^*$-inclusions is completely settled in the affirmative (Corollary
\ref{W^* ARC}).
\end{remark}

\subsection{Characterization: Unique Pseudo-Expectation Property for
  Abelian Subalgebras}

In this section we give an order-theoretic characterization of the
unique pseudo-expectation property for $C^*$-inclusions $(\C,\D)$,
with $\D$ \underline{abelian}. We remind the reader that if $\D$ is a
unital abelian $C^*$-algebra, then $I(\D)$ is \emph{order complete},
meaning that every nonempty set $S \subseteq I(\D)_{sa}$ with an upper
bound has a supremum \cite[Prop. III.1.7]{Takesaki1979}.

\begin{theorem} \label{unique characterization}
Let $(\C,\D)$ be a $C^*$-inclusion, with $\D$ abelian. Then the following are equivalent:
\begin{enumerate}
\item $(\C,\D)$ has the unique pseudo-expectation property.
\item For all $x \in \C_{sa}$,
\[
    \sup_{I(\D)}\{d \in \D_{sa}: d \leq x\} = \inf_{I(\D)}\{d \in \D_{sa}: d \geq x\}.
\]
\end{enumerate}
\end{theorem}

\begin{proof} For $x \in \C_{sa}$, set
\[ \ell(x) = \sup_{I(\D)}\{d \in \D_{sa}: d \leq x\} \text{ and } u(x)
= \inf_{I(\D)}\{d \in \D_{sa}: d \geq x\}.
\] It is easy to see that $\ell(x) \leq u(x)$. (If $e \in \D_{sa}$ and
$e \geq x$, then $d \leq e$ for all $d \in \D_{sa}$ such that $d \leq
x$. Thus $\ell(x) \leq e$. Since the choice of $e$ was arbitrary,
$\ell(x) \leq u(x)$.) Clearly $\ell(d) = d = u(d)$ for all $d \in
\D_{sa}$.

(i $\implies$ ii) Let $x \in \C_{sa} \backslash \D_{sa}$ and suppose
$a \in I(\D)_{sa}$ satisfies $\ell(x) \leq a \leq u(x)$. Since $\D
\cap \bbC x = 0$,
\[ \Phi_0:\D + \bbC x \to I(\D): d + \lambda x \mapsto d + \lambda a
\] is a well-defined linear map such that $\Phi_0|_{\D} =
\id$. Suppose $d + \lambda x \geq 0$, so that $d \in \D_{sa}$ and
$\lambda \in \bbR$.
\begin{enumerate}
\item[Case 1:] If $\lambda = 0$, then $d \geq 0$, which implies $d +
\lambda a = d \geq 0$.
\item[Case 2:] If $\lambda > 0$, then $x \geq -\frac{1}{\lambda}d$,
which implies $-\frac{1}{\lambda}d \leq \ell(x) \leq a$, which in turn
implies $d + \lambda a \geq 0$.
\item[Case 3:] If $\lambda < 0$, then $x \leq -\frac{1}{\lambda}d$,
which implies $a \leq u(x) \leq -\frac{1}{\lambda}d$, which in turn
implies $d + \lambda a \geq 0$.
\end{enumerate} The preceding analysis shows that $\Phi_0$ is
positive, and since $I(\D)$ is abelian, it is actually completely
positive. By injectivity, there exists a ucp map $\Phi:\C \to I(\D)$
such that $\Phi|_{\D + \bbC x} = \Phi_0$. Then $\Phi$ is a
pseudo-expectation for $(\C,\D)$ such that $\Phi(x) = a$. It follows
that if there exists $x \in \C_{sa}$ such that $\ell(x) \neq u(x)$,
then $(\C,\D)$ admits multiple pseudo-expectations.

(ii $\implies$ i) Conversely, suppose $\Phi \in \PsExp(\C,\D)$. Let $x
\in \C_{sa}$. If $d \in \D_{sa}$ and $d \leq x$, then $d = \Phi(d)
\leq \Phi(x)$, which implies $\ell(x) \leq \Phi(x)$. Likewise if $d
\in \D_{sa}$ and $d \geq x$, then $d = \Phi(d) \geq \Phi(x)$, which
implies $u(x) \geq \Phi(x)$. Thus if $\ell(x) = u(x)$ for all $x \in
\C_{sa}$, then $\Phi$ is uniquely determined on $\C_{sa}$, therefore
on $\C$.
\end{proof}

\subsection{A Krein-Milman Theorem for Pseudo-Expectations
 when the Subalgebra is
  Abelian}

In this section we prove a Krein-Milman theorem for the
pseudo-expectation space $\PsExp(\C,\D)$, valid for $C^*$-inclusions
$(\C,\D)$, with $\D$ \underline{abelian}.  Our goal is to show that
there is a rich supply of extreme points in $\PsExp(\C,\D)$.  It will
then follow that uniqueness of pseudo-expectations is equivalent to
uniqueness of extreme pseudo-expectations.  One approach to this type
of result
might be the following: first,  introduce an appropriate locally convex topology on the
set of all bounded linear maps from $\C$ into $I(\D)$; second, show
that $\PsExp(\C,\D)$ is compact in this topology; and finally, apply the usual
Krein-Milman theorem.  While this may be a viable
approach, it is not clear (at least to us) how to define such a
topology, so we proceed instead using a route through convexity
theory, which is perhaps less well-traveled.

Our key tool is Kutateladze's Krein-Milman theorem for
subdifferentials of sublinear operators into Kantorovich spaces
\cite{Kutateladze1980}. Let $V$ and $W$ be real vector spaces. Assume
further that $W$ is a \emph{Kantorovich space}, meaning that $W$ is a
vector lattice such that every nonempty subset with an upper bound has
a supremum. Suppose $Q:V \to W$ a \emph{sublinear operator}, meaning
that
\begin{itemize}
\item $Q(\alpha v) = \alpha Q(v)$ for all $v \in V$, $\alpha \geq 0$;
\item $Q(v_1 + v_2) \leq Q(v_1) + Q(v_2)$ for all $v_1, v_2 \in V$.
\end{itemize}
Let $\partial Q$ be the \emph{subdifferential} of $Q$:
\[
    \partial Q = \{T \in \Lin(V,W): T(v) \leq Q(v), ~ v \in V\}.
\]
(Here $\Lin(V,W)$ denotes the set of all real linear maps from $V$ to
$W$.)  Kutateladze's version of the Krein-Milman theorem is the following.

\begin{theorem}[Kutateladze~\cite{Kutateladze1980}]  \label{Kutateladze}  Let $V$ and $W$ be
  real vector spaces with $W$ a Kantorovich space, and suppose
  $Q:V\rightarrow W$ is a sublinear operator. Then the following
  statements hold:
\begin{enumerate}
\item $\Ext(\partial Q) \neq \emptyset.$
\item  For $v\in V$, define 
$    P(v) = \sup_W\{T(v): T \in \Ext(\partial Q)\}$.
Then $P:V\rightarrow W$ is a sublinear operator and
   $ \partial Q = \partial P.$
 \end{enumerate}
\end{theorem}

We are now ready to apply Kutateladze's Theorem to our setting.

\begin{theorem} \label{KM}
Let $(\C,\D)$ be a $C^*$-inclusion, with $\D$ abelian. Then the following
statements hold:
\begin{enumerate}
\item
    $\Ext(\PsExp(\C,\D)) \neq \emptyset$.
\item For  $x \in \C_{sa}$, define
    $P(x) = \sup_{I(\D)}\{\Psi(x): \Psi \in \Ext(\PsExp(\C,\D))\}$.
    Then
\[
\PsExp(\C,\D) = \{\Phi \in \UCP(\C,I(\D)): \Phi(x) \leq P(x) \text{ for
  all } x \in
\C_{sa}\}.
\]
\end{enumerate}
In particular,
\[
\exists! ~ \Phi \in \PsExp(\C,\D) \iff \exists! ~ \Psi \in
\Ext(\PsExp(\C,\D)).
\]
\end{theorem}

\begin{proof}
  Since $\D$ is abelian, $I(\D)_{sa}$ is a Kantorovich space. For all
  $x \in \C_{sa}$, define
\[
    Q(x) = \sup_{I(\D)}\{\Phi(x): \Phi \in \PsExp(\C,\D)\}.
\] It is easy to see that $Q:\C_{sa} \to I(\D)_{sa}$ is a sublinear
operator. We claim that
\[ \PsExp(\C,\D) = \{\Phi \in \UCP(\C,I(\D)): \Phi|_{\C_{sa}}
\in \partial Q\} = \{\tilde{T}: T \in \partial Q\},
\] where $\tilde{T}:\C \to I(\D)$ is the \emph{complexification} of
$T:\C_{sa} \to I(\D)_{sa}$. Indeed, the inclusions of the first set
into the second, and the second set into the third, are
tautological. Now let $T \in \partial Q$. Then
\[ x \in \C_+ \implies -T(x) = T(-x) \leq Q(-x) \leq 0 \implies T(x)
\geq 0.
\] Thus $\tilde{T}$ is positive, and since $I(\D)$ is abelian,
completely positive. Also
\[ d \in \D_{sa} \implies \pm T(d) = T(\pm d) \leq Q(\pm d) = \pm d
\implies T(d) = d.
\] Therefore, $\tilde{T}
\in \PsExp(\C,\D)$.

Invoking Kutateladze's Krein-Milman theorem, we have that
\[ \Ext(\PsExp(\C,\D)) \neq \emptyset,
\] and in fact
\[ \PsExp(\C,\D) = \{\Phi \in \UCP(\C,I(\D)): \Phi|_{\C_{sa}}
\in \partial P\},
\] where for all $x \in \C_{sa}$,
\[ P(x) = \sup_{I(\D)}\{\tilde{T}(x): T \in \Ext(\partial Q)\} =
\sup_{I(\D)}\{\Psi(x): \Psi \in \Ext(\PsExp(\C,\D))\}.
\]

If there exists a unique $\Phi \in \PsExp(\C,\D)$, then clearly there
exists a unique $\Psi \in \Ext(\PsExp(\C,\D))$. Conversely, suppose
there exists a unique $\Psi \in \Ext(\PsExp(\C,\D))$. Then for all
$\Phi \in \PsExp(\C,\D)$,
\[ x \in \C_{sa} \implies \pm \Phi(x) = \Phi(\pm x) \leq P(\pm x) =
\Psi(\pm x) = \pm \Psi(x) \implies \Phi(x) = \Psi(x),
\] and so $\Phi = \Psi$.
\end{proof}

\subsection{Abelian Inclusions} \label{abelian section}

In this section we consider the unique pseudo-expectation property for
$C^*$-inclusions $(\A,\D)$, with $\A$ abelian. By Gelfand duality,
these are precisely the $C^*$-inclusions $(C(Y),C(X))$, where $X$ and
$Y$ are compact Hausdorff spaces. We recall that unital
$*$-monomorphisms $\pi:C(X) \to C(Y)$ correspond bijectively to
continuous surjections $j:Y \to X$. Indeed, if $j:Y \to X$ is a
continuous surjection, then $\pi_j:C(X) \to C(Y):f \mapsto f \circ j$
is a unital $*$-monomorphism. We may identify $\pi_j(C(X))$ with the
continuous functions on $Y$ which are constant on the fibers
$j^{-1}(x)$, $x \in X$. Conversely, if $\pi:C(X) \to C(Y)$ is a unital
$*$-monomorphism, then for each $y \in Y$ there exists a unique $j(y)
\in X$ such that $\delta_y \circ \pi = \delta_{j(y)}$, and it is easy
to verify that $j:Y \to X$ is a continuous surjection such that $\pi_j
= \pi$.

In light of Theorem \ref{KM}, to characterize when $\PsExp(\A,\D)$ is
a singleton, it suffices to characterize when $\Ext(\PsExp(\A,\D))$ is
a singleton. As we saw in Proposition \ref{basic properties PsExp},
\[ \Ext(\PsExp(\A,\D)) \subseteq \Ext(\UCP(\A,I(\D))).
\] On the other hand, $\Psi \in \Ext(\UCP(\A,I(\D)))$ iff $\Psi$ is
multiplicative (i.e., a unital $*$-homomorphism)
\cite[Cor. 3.1.6]{Stormer2013}. Thus the extreme pseudo-expectations
for $(\A,\D)$ are precisely the \emph{multiplicative
pseudo-expectations}:
\[ \Ext(\PsExp(\A,\D)) = \PsExp^\times(\A,\D).
\]

\begin{theorem} \label{mult_PsExp} Let $(\A,\D)$ be an abelian
inclusion. Then the mapping
\[ \PsExp^\times(\A,\D) \to \left\{\substack{\text{maximal
$\D$-disjoint}\\ \text{ideals of $\A$}}\right\} \quad\text{given
by}\quad
\Psi \mapsto \ker(\Psi)
\] is a bijection. In particular, $\PsExp^\times(\A,\D)$ is a
singleton iff there exists a unique maximal $\D$-disjoint ideal $\I
\ideal \A$.
\end{theorem}

\begin{proof} Let $\Psi \in \PsExp^\times(\A,\D)$. Then $\ker(\Psi)$
is a $\D$-disjoint ideal of $\A$, and the map $\I \mapsto \Psi(\I)$ is
an order-preserving bijection between the $\D$-disjoint ideals of $\A$
containing $\ker(\Psi)$ and the $\D$-disjoint ideals of
$\Psi(\A)$. Since $(I(\D),\D)$ has the faithful unique
pseudo-expectation property (Example \ref{OSE}), it is hereditarily
essential, by Theorem \ref{faithful characterization}. Thus
$(\Psi(\A),\D)$ is essential, so that the only $\D$-disjoint ideal of
$\Psi(\A)$ is $0$. It follows that the only $\D$-disjoint ideal of
$\A$ containing $\ker(\Psi)$ is $\ker(\Psi)$ itself, which says that
$\ker(\Psi)$ is a maximal $\D$-disjoint ideal of $\A$.

Conversely, suppose $\I \subseteq \A$ is a maximal $\D$-disjoint
ideal. The map $\Psi_0:\I + \D \to \D: h + d \mapsto d$ is a unital
$*$-homomorphism. Since $I(\D)$ is an injective unital abelian
$C^*$-algebra, there exists a unital $*$-homomorphism $\Psi:\A \to
I(\D)$ such that $\Psi|_{\I + \D} = \Psi_0$
\cite{HadwinPaulsen2011}. Clearly $\Psi \in \PsExp^\times(\A,\D)$ and
$\I \subseteq \ker(\Psi) \subseteq \A$ is a $\D$-disjoint ideal. Thus
$\ker(\Psi) = \I$, by maximality.

Finally, suppose $\Psi_1, \Psi_2 \in \PsExp^\times(\A,\D)$, with
$\ker(\Psi_1) = \ker(\Psi_2)$. Define $\iota:\Psi_1(\A) \to
\Psi_2(\A)$ by the formula $\iota(\Psi_1(x)) = \Psi_2(x)$, $x \in
\A$. Then $\iota$ is a unital $*$-isomorphism which fixes $\D$. By
injectivity, there exists an unital $*$-homomorphism
$\overline{\iota}:I(\D) \to I(\D)$ such that
$\overline{\iota}|_{\Psi_1(\A)} = \iota$. By the rigidity of the
injective envelope, $\overline{\iota} = \id$, so that $\Psi_1 =
\Psi_2$.
\end{proof}

\begin{remark} Taking $\D = \bbC$ in Theorem \ref{mult_PsExp} above,
one recovers the well-known bijective correspondence between the
characters and the maximal ideals of $\A$.
\end{remark}

\begin{corollary} \label{abelian} Let $(C(Y),C(X))$ be an abelian
inclusion with corresponding continuous surjection $j:Y \to X$. Then
the following are equivalent:
\begin{enumerate}
\item There exists a unique pseudo-expectation for $(C(Y),C(X))$.
\item There exists a unique multiplicative pseudo-expectation for
$(C(Y),C(X))$.
\item There exists a unique maximal $C(X)$-disjoint ideal in
$C(Y)$.
\item There exists a unique minimal closed set $K \subseteq Y$
such that $j(K) = X$.
\end{enumerate}
\end{corollary}

\begin{proof} (i $\iff$ ii) follows from Theorem \ref{KM}.

(ii $\iff$ iii) Theorem \ref{mult_PsExp}.

(iii $\iff$ iv) The map $K \mapsto \{g \in C(Y): g|_K = 0\}$ is an
order-reversing bijection between the closed sets $K \subseteq Y$ such
that $j(K) = X$ and the $C(X)$-disjoint ideals in $C(Y)$.
\end{proof}

\begin{corollary} \label{faithful abelian} Let $(C(Y),C(X))$ be an
abelian inclusion with corresponding continuous surjection $j:Y \to
X$. Then the following are equivalent:
\begin{enumerate}
\item There exists a unique pseudo-expectation for $(C(Y),C(X))$,
which is faithful.
\item There exists a unique multiplicative pseudo-expectation for
$(C(Y),C(X))$, which is faithful.
\item $(C(Y),C(X))$ is essential.
\item If $K \subseteq Y$ is closed and $j(K) = X$, then $K = Y$.
\end{enumerate}
\end{corollary}

\begin{proof} Same proof as Corollary \ref{abelian}.
\end{proof}

\section{Examples} \label{examples section}

Now we provide additional examples of $C^*$-inclusions with (and
without) the unique pseudo-expectation property (resp. the faithful
unique pseudo-expectation property). In proving that these examples
are actually examples, we will take advantage of some of the general
theory developed so far.

In Section \ref{basics section}, we mentioned that the unique
pseudo-expectation property is not hereditary from below. Equipped
with the results of the previous section, it is easy to give an example which
demonstrates this.

\begin{example} \label{not hereditary from below} There exist abelian
inclusions $\D \subseteq \D_0 \subseteq \A$ such that $(\A,\D)$ has
the unique pseudo-expectation property, but $(\A,\D_0)$ does not. That
is, the unique pseudo-expectation property is not hereditary from
below.
\end{example}

\begin{proof} Let $X = [0,1]$, $X_0 = [0,1] \cup \{2\}$, and $Y =
[0,1] \cup \{2, 3\}$. Define continuous surjections $j:X_0 \to X$ and
$k:Y \to X_0$ by the formulas
\[ j(t) = \begin{cases} t, & t \in [0,1]\\ 1, & t = 2 \end{cases}
\text{ and } k(t) = \begin{cases} t, & t \in [0,1]\\ 2, & t \in \{2,
3\} \end{cases}.
\] Then $i = j \circ k:Y \to X$ is the continuous surjection
\[ i(t) = \begin{cases} t, & t \in [0,1]\\ 1, & t \in \{2,
3\} \end{cases}.
\] Clearly there exists a unique minimal closed set $K \subseteq Y$
such that $i(K) = X$, namely $K = [0,1]$. On the other hand, there are
multiple minimal closed sets $L \subseteq Y$ such that $k(L) = X_0$,
for example both $L = [0,1] \cup \{2\}$ and $L = [0,1] \cup
\{3\}$. Thus by Corollary \ref{abelian} we have inclusions $C(X)
\subseteq C(X_0) \subseteq C(Y)$ such that $(C(Y),C(X))$ has the
unique pseudo-expectation property, but $(C(Y),C(X_0))$ does not.
\end{proof}

We can also provide examples of the poor behavior of the unique
pseudo-expectation property with respect to quotients described in
Section \ref{quotient section}. To that end, let $(C(Y),C(X))$ be an
abelian inclusion, with corresponding continuous surjection $j:Y \to
X$. Suppose $Z \subseteq Y$ is closed and $\J = \{g \in C(Y): g|_Z =
0\} \ideal C(Y)$. Then $\J \cap C(X) = \{f \in C(X): f|_{j(Z)} = 0\}
\ideal C(X)$. Thus
\[ (C(Y)/\J,C(X)/(\J \cap C(X))) \cong (C(Z),C(j(Z)),
\] with corresponding continuous surjection $j|_Z:Z \to
j(Z)$. Furthermore $\J \cap C(X) \ideal C(X)$ is regular iff
$\o{j(Z)^\circ} = j(Z)$, where the interior and closure are calculated
in $X$.

\begin{example} \label{non-unique quotient} There exists an abelian
inclusion $(\A,\D)$ and $\J \ideal \A$ such that $(\A,\D)$ has the
unique pseudo-expectation property, but $(\A/\J,\D/(\J \cap \D))$ does
not. Of course, $\J \cap \D \ideal \D$ is not a regular ideal.
\end{example}

\begin{proof} Let $Y = ([0,1] \times \{0\}) \cup (\{1\} \times [0,1])
\subseteq [0,1] \times [0,1]$, $X = [0,1]$, and $j:Y \to X$ be defined
by the formula $j(s,t) = s$, $(s,t) \in Y$. Then there exists a unique
minimal closed set $K \subseteq Y$ such that $j(K) = X$, namely $K =
[0,1] \times \{0\}$. Thus $(C(Y),C(X))$ has the unique
pseudo-expectation property, by Corollary \ref{abelian}. Now let $Z =
\{1\} \times [0,1]$, a closed subset of $Y$. Then $j(Z) = \{1\}$, and
there does not exist a unique minimal closed set $L \subseteq Z$ such
that $j(L) = j(Z)$. Thus $(C(Z),C(j(Z))$ does not have the unique
pseudo-expectation property. Of course $\o{j(Z)^\circ} = \emptyset
\subsetneq j(Z)$.
\end{proof}

\begin{example} \label{unfaithful quotient} There exists an abelian
  inclusion $(\A,\D)$ and $\J \ideal \A$ such that $(\A,\D)$ has the
  faithful unique pseudo-expectation property and $\J \cap \D$ is a
  regular ideal in $\D$, but $(\A/\J,\D/(\J \cap \D))$ does not have
  the faithful unique pseudo-expectation property.
\end{example}

\begin{proof} Let $Y = ([0,1/2] \times \{0\}) \cup ([1/2,1] \times
\{1\}) \subseteq [0,1] \times [0,1]$, $X = [0,1]$, and $j:Y \to X$ be
defined by the formula $j(s,t) = s$, $(s,t) \in Y$. Then there exists
a unique minimal closed set $K \subseteq Y$ such that $j(K) = X$,
namely $K = Y$. Thus $(C(Y),C(X))$ has the faithful unique
pseudo-expectation property, by Corollary \ref{faithful abelian}. Now
let $Z = ([0,1/2] \times \{0\}) \cup \{(1/2,1)\}$, a closed subset of
$Y$. Then $j(Z) = [0,1/2]$, so that $\o{j(Z)^\circ} = j(Z)$. There
exists a unique minimal closed set $L \subseteq Z$ such that $j(L) =
j(Z)$, namely $L = [0,1/2] \times \{0\}$. Since $L \subsetneq Z$,
$(C(Z),C(j(Z)))$ has the unique pseudo-expectation property, but not
the faithful unique pseudo-expectation property.
\end{proof}

In the introduction we mentioned that the inclusion
$(B(L^2[0,1]),C[0,1])$ admits no conditional expectations (Example
\ref{!CE not hereditary from below}). An interesting question (posed
to us by Philip Gipson) is whether or not $(B(L^2[0,1]),C[0,1])$ has a
unique pseudo-expectation. It turns out that even the abelian
inclusion $(L^\infty[0,1],C[0,1])$ admits multiple
pseudo-expectations. We found it difficult to fit this example into
the context of Corollary~\ref{abelian}, so we utilize
Theorem~\ref{unique characterization} instead.

\begin{example} The abelian inclusion $(L^\infty[0,1],C[0,1])$ has
infinitely many pseudo-expectations, none of which are faithful.
\end{example}

\begin{proof} Let $B^\infty[0,1]$ be the $C^*$-algebra of bounded
complex-valued Borel functions on $[0,1]$. Let $N[0,1] \ideal
B^\infty[0,1]$ be the Lebesgue-null functions, so that
$B^\infty[0,1]/N[0,1] = L^\infty[0,1]$. Likewise, let $M[0,1] \ideal
B^\infty[0,1]$ be the meager functions, so that $B^\infty[0,1]/M[0,1]
= D[0,1]$, the \emph{Dixmier algebra}. Recall that $D[0,1] =
I(C[0,1])$ \cite{HadwinPaulsen2011}.

Now let $A \subseteq [0,1]$ be a Borel set such that both $A$ and
$A^c$ are measure dense, meaning that $|V \cap A| > 0$ and $|V \cap
A^c| > 0$ for every open set $V \subseteq [0,1]$. (Here $|\cdot|$
stands for Lebesgue measure.) One possible construction of $A$ can be
found in \cite{Rudin1983}.

A measure-theoretic argument shows that if $f \in C[0,1]_{sa}$ and $f
+ N[0,1] \leq \chi_A + N[0,1]$, then $f \leq 0$.
Likewise,
if $f \in C[0,1]_{sa}$ and $f + N[0,1] \geq \chi_A + N[0,1]$, then $f
\geq 1$.

It follows that
\[ \sup_{D[0,1]}\{f + M[0,1]: f \in C[0,1]_{sa}, ~ f + N[0,1] \leq
\chi_A + N[0,1]\} = 0 + M[0,1]
\] and
\[ \inf_{D[0,1]}\{f + M[0,1]: f \in C[0,1]_{sa}, ~ f + N[0,1] \geq
\chi_A + N[0,1]\} = 1 + M[0,1].
\] By Theorem \ref{unique characterization}, $(L^\infty[0,1],C[0,1])$
does not have the unique pseudo-expectation property.

It remains to show that no pseudo-expectation for
$(L^\infty[0,1],C[0,1])$ is faithful. By
\cite[Thm. 2.21]{HadwinPaulsen2011}, there are $C^*$-inclusions
$C[0,1] \subseteq I(C[0,1]) \subseteq L^\infty[0,1]$. Suppose there
exists a faithful pseudo-expectation $\Phi \in
\PsExp(L^\infty[0,1],C[0,1])$. Then $\Phi:L^\infty[0,1] \to I(C[0,1])$
is a ucp map such that $\Phi|_{C[0,1]} = \id$. By the rigidity of the
injective envelope, $\Phi|_{I(C[0,1])} = \id$. Thus $\Phi$ is a
faithful conditional expectation of $L^\infty[0,1]$ onto
$I(C[0,1])$. It follows from \cite[Lemma 1]{Kadison1956} that $D[0,1]
= I(C[0,1])$ is a $W^*$-algebra, contradicting \cite[Exercise
5.7.21]{KadisonRingrose1983}.
\end{proof}

\subsubsection*{\textsc{Transformation Group
    $C^*$-Algebras.}} \label{transformation section} Let $\Gamma$ be a
discrete group acting on a compact Hausdorff space $X$ by
homeomorphisms, and let $C(X) \rtimes_r \Gamma$ (resp. $C(X) \rtimes
\Gamma$) be the corresponding reduced (resp. full) crossed product
(see \cite[Ch. 4]{BrownOzawa2008} for more details). In this section
we examine the unique pseudo-expectation property for the inclusions
$(C(X) \rtimes_r \Gamma,C(X))$ and $(C(X) \rtimes_r
\Gamma,C(X)^c)$. We recall that elements of $C(X) \rtimes_r \Gamma$
have formal series representations $\sum_{t \in \Gamma} a_t\lambda_t$,
where $a_t \in C(X)$ for all $t \in \Gamma$, and that there exists a
faithful conditional expectation $E:C(X) \rtimes_r \Gamma \to C(X)$,
namely
\[ E\left(\sum_{t \in \Gamma} a_t\lambda_t\right) = a_e.
\]

For $s\in \Gamma$, write $F_s=\{x\in X: sx=x\}$ for the \emph{fixed points}
of $s$,  and for
$x\in X$, let \[H^x:=\{s\in \Gamma: x\in (F_s)^\circ\}.\]
The condition that $H^x$ is abelian for every $x\in X$ is equivalent
to the condition that $C(X)^c$ is
abelian~\cite[Theorem~6.6]{Pitts2012}.  The following result shows
that when either of these equivalent conditions hold, then $(C(X) \rtimes_r
\Gamma,C(X)^c)$ has the faithful unique pseudo-expectation property.
\begin{proposition}[{\cite[Theorem~6.10]{Pitts2012}}] \label{relative commutant crossed} If $C(X)^c$ is
abelian, then $(C(X) \rtimes_r \Gamma,C(X)^c)$ has the faithful unique pseudo-expectation property. In particular, this happens when $\Gamma$ is abelian.
\end{proposition}

In general, we do not know a characterization of when $(C(X)\rtimes_r
\Gamma, C(X)^c)$ has the faithful unique pseudo-expectation property,
or even the unique pseudo-expectation property.
However, we do have the following result for the inclusion
$(C(X)\rtimes_r\Gamma, C(X))$.

\begin{theorem} \label{crossed} If $\Gamma$ is amenable (more
generally, if $C(X) \rtimes_r \Gamma = C(X) \rtimes \Gamma$), then the
following are equivalent:
\begin{enumerate}
\item $(C(X) \rtimes_r \Gamma,C(X))$ has the unique
pseudo-expectation property.
\item $(C(X) \rtimes_r \Gamma,C(X))$ has the faithful unique
pseudo-expectation property.
\item $(C(X) \rtimes_r \Gamma,C(X))$ is essential.
\item The action of $\Gamma$ on $X$ is topologically free (i.e.,
$(F_t)^\circ = \emptyset$ for all $e \neq t \in \Gamma$).
\item $C(X)^c = C(X)$ (i.e., $C(X)$ is a MASA in $C(X) \rtimes_r
\Gamma$).
\end{enumerate}
\end{theorem}

\begin{proof} (i $\implies$ ii) If $(C(X) \rtimes_r \Gamma,C(X))$ has
the unique pseudo-expectation property, then $\PsExp(C(X) \rtimes_r
\Gamma,C(X)) = \{E\}$, and $E$ is faithful.

(ii $\implies$ iii) Theorem \ref{faithful characterization}.

(iii $\implies$ iv) \cite[Thm. 4.1]{KawamuraTomiyama1990}.

(iv $\implies$ v) By \cite[Prop. 6.3]{Pitts2012} and the topological
freeness of the action of $\Gamma$ on $X$, we have that
\begin{eqnarray*} C(X)^c &=& \left\{\sum_{t \in \Gamma} a_t\lambda_t
\in C(X) \rtimes_r \Gamma: \supp(a_t) \subseteq (F_t)^\circ, ~ t \in
\Gamma\right\}\\ &=& \left\{\sum_{t \in \Gamma} a_t\lambda_t \in C(X)
\rtimes_r \Gamma: a_t = 0, ~ e \neq t \in \Gamma\right\} = C(X).
\end{eqnarray*}

(v $\implies$ i) Since $C(X)^c = C(X)$, $(C(X) \rtimes_r \Gamma,C(X))$
is a regular MASA inclusion. Thus $(C(X) \rtimes_r \Gamma,C(X))$ has
the unique pseudo-expectation property, by Pitts' Theorem \ref{Pitts}.
\end{proof}

\begin{remark} Most of the implications in Theorem \ref{crossed}
remain valid in full generality (without the assumption $C(X)
\rtimes_r \Gamma = C(X) \rtimes \Gamma$). In particular,
\[ \text{(iv)} \iff \text{(v)} \implies \text{(i)} \iff \text{(ii)} \implies
\text{(iii)}.
\]
\end{remark}

\section{$W^*$-Inclusions} \label{W^* section}

In this section we investigate the unique pseudo-expectation property
for $W^*$-inclusions $(\M,\D)$. This means that $(\M,\D)$ is a
\cstar-inclusion such that $\M$  is a $W^*$-algebra and $\D$ is
$\sigma(\M,\M_*)$-closed.

First we consider abelian $W^*$-inclusions. Corollary \ref{abelian}
above shows that there exist nontrivial abelian $C^*$-inclusions
$(C(Y),C(X))$ with the unique pseudo-expectation property. Not so for
abelian $W^*$-inclusions, due to the abundance of normal states.

\begin{theorem} \label{abelian W^*}  Let $(\M,\D)$ be a $W^*$-inclusion.
\begin{enumerate}
\item Suppose $\M$ is abelian. Then
$(\M,\D)$ has the unique pseudo-expectation property iff $\M = \D$.
\item More generally, let $(\M,\D)$ be a $W^*$-inclusion, with
  $\D$ abelian (and $\M$ possibly non-abelian). If $(\M,\D)$ has the
  unique pseudo-expectation property, then $\D$ is a MASA in $\M$.
\end{enumerate}
\end{theorem}
Observe that because $\D$ is an abelian von
Neumann algebra, the pseudo-expectations  in the theorem are
conditional expectations.

\begin{proof} (i) Suppose $\PsExp(\M,\D) = \{E\}$. Let $a \in
\M_{sa}$. Since $\M$ is abelian,
\[ \{d \in \D_{sa}: d \leq a\}
\] is an increasing net indexed by itself. Indeed, if $f, g \leq h$
are continuous functions, then $\max\{f,g\} \leq h$. By Theorem
\ref{unique characterization}, we have that
\[ E(a) = \sup_{\D}\{d \in \D_{sa}: d \leq a\}.
\] Now let $\phi \in (\D_*)_+$ and $\overline{\phi} \in (\M^*)_+$ be
an extension. Then
\[ \phi(E(a)) = \sup\{\phi(d): d \leq a\},
\] by normality. On the other hand, if $d \leq a$, then $\phi(d) =
\overline{\phi}(d) \leq \overline{\phi}(a)$, which implies $\phi(E(a))
\leq \overline{\phi}(a)$. Replacing $a$ by $-a$, we conclude that
$\phi(E(a)) = \overline{\phi}(a)$, and so $\overline{\phi} = \phi
\circ E$. Thus if $a \in \M$ and $\psi \in (\M_*)_+$, then $\psi =
\psi|_{\D} \circ E$, which implies $\psi(a) = \psi(E(a))$. Since the
choice of $\psi$ was arbitrary, $a = E(a) \in \D$.

(ii) Let $\D \subseteq \A \subseteq \M$ be a MASA. Since $(\M,\D)$ has
the unique pseudo-expectation property, so does $(\A,\D)$, by
Proposition \ref{hereditary from above}. Then $\A = \D$, by (i) above.
\end{proof}

Our next objective is to generalize Theorem \ref{abelian W^*}, by
showing that for an arbitrary $W^*$-inclusion $(\M,\D)$, the unique
pseudo-expectation property implies that $\D^c = Z(\D)$ (Corollary
\ref{W^* ARC}). Our proof relies on a nice bijective correspondence
between the conditional expectations $\D^c \to Z(\D)$ and the
conditional expectations $C^*(\D,\D^c) \to \D$ (Theorem \ref{CE
  bijection}).  Theorem~\ref{CE bijection} is related to
\cite[Thm. 5.3]{CombesDelaroche1975}, but to our knowledge, is new.
Our proof of Theorem~\ref{CE bijection} uses some fairly recent
technology, which we now describe.

Let $\M$, $\N$ be $W^*$-algebras and $\Z \subseteq Z(\M) \cap Z(\N)$
be a $W^*$-subalgebra. By \cite{Blanchard1995,GiordanoMingo1997},
there exist on the $\Z$-balanced algebraic tensor product
$\M \otimes_{\Z} \N$ both a minimal $C^*$-norm $\|\cdot\|_{\min}$ and
a maximal $C^*$-norm $\|\cdot\|_{\max}$, which coincide if either $\M$
or $\N$ is abelian.   When $\Z=\bbC$, this fact is now classical,
see~\cite[Chapter~IV.4]{Takesaki1979}.  Now suppose that for $i = 1, 2$, $\M_i$, $\N_i$
are $W^*$-algebras and $\Z \subseteq Z(\M_i) \cap Z(\N_i)$ is a
$W^*$-subalgebra. If $u:\M_1 \to \M_2$ and $v:\N_1 \to \N_2$ are
completely contractive $\Z$-bimodule maps, then the unique
$\Z$-bimodule map
$u \otimes v:(\M_1 \otimes_{\Z} \N_1,\|\cdot\|_{\min}) \to (\M_2
\otimes_{\Z} \N_2,\|\cdot\|_{\min})$
such that $(u \otimes v)(x \otimes y) = u(x) \otimes v(y)$ for all
$x \in \M_1$, $y \in \N_1$ is a contraction. Furthermore, if
$\M_1 \subseteq \M_2$ and $\N_1 \subseteq \N_2$, then
$(\M_1 \otimes_{\Z} \N_1,\|\cdot\|_{\min}) \subseteq (\M_2
\otimes_{\Z} \N_2,\|\cdot\|_{\min})$.

\begin{theorem} \label{CE bijection} Let $(\M,\D)$ be a
$W^*$-inclusion. Then the map $E \mapsto E|_{\D^c}$ is a bijective
correspondence between the conditional expectations $C^*(\D,\D^c) \to
\D$ and the conditional expectations $\D^c \to Z(\D)$. In particular,
there exists a conditional expectation $C^*(\D,\D^c) \to \D$.
\end{theorem}

\begin{proof} Let $E:C^*(\D,\D^c) \to \D$ be a conditional
expectation. For all $d' \in \D^c$ and $d \in \D$, we have that
\[ dE(d') = E(dd') = E(d'd) = E(d')d,
\] which implies $E(d') \in Z(\D)$. Conversely, suppose $\theta:\D^c
\to Z(\D)$ is a conditional expectation. By
\cite[Thm. 5.5.4]{KadisonRingrose1983},
\[ \D \otimes_{Z(\D)} \D^c \to \Alg(\D,\D^c) \subseteq \M:
\sum_{i=1}^n d_i \otimes d_i' \mapsto \sum_{i=1}^n d_id_i'
\] is a $*$-isomorphism. Thus (with the notation of the previous
paragraph)
\[ \|\sum_{i=1}^n d_i \otimes d_i'\|_{\min} \leq \|\sum_{i=1}^n d_id_i'\|
\leq \|\sum_{i=1}^n d_i \otimes d_i'\|_{\max}, ~ \sum_{i=1}^n d_i \otimes
d_i' \in \D \otimes_{Z(\D)} \D^c.
\] Furthermore, since $Z(\D)$ is abelian,
\[ \|\sum_{i=1}^n d_i \otimes z_i\|_{\min} = \|\sum_{i=1}^n d_iz_i\|, ~
\sum_{i=1}^n d_i \otimes z_i \in \D \otimes_{Z(\D)} Z(\D).
\] Thus for all $\sum_{i=1}^n d_i \otimes d_i' \in \D \otimes_{Z(\D)}
\D^c$,
\[ \|\sum_{i=1}^n d_i\theta(d_i')\| = \|\sum_{i=1}^n d_i \otimes
\theta(d_i')\|_{\min} \leq \|\sum_{i=1}^n d_i \otimes d_i'\|_{\min} \leq
\|\sum_{i=1}^n d_id_i'\|.
\] It follows that the map
\[ \Alg(\D,\D^c) \to \Alg(\D,Z(\D)) = \D: \sum_{i=1}^n d_id_i' \mapsto
\sum_{i=1}^n d_i\theta(d_i')
\] extends uniquely to a conditional expectation $\Theta:C^*(\D,\D^c)
\to \D$. Clearly the maps $E \mapsto E|_{\D^c}$ and $\theta \mapsto
\Theta$ described above are inverse to one another.
\end{proof}

Recall that in the case of a \cstar-inclusion,  the
\underline{faithful} unique pseudo-expectation property implies that $\D^c$ is
abelian, but we do not know whether the faithfulness assumption can be
dropped.  However, the following corollary to Theorem~\ref{CE
  bijection} shows that in the $W^*$-case, faithfulness is not
necessary to conclude $\D^c$ is abelian.  In fact, more is true.
\begin{theorem} \label{W^* ARC} Let $(\M,\D)$ be a
$W^*$-inclusion. If $(\M,\D)$ has the unique pseudo-expectation
property, then $\D^c = Z(\D)$.
\end{theorem}

\begin{proof} If $(\M,\D)$ has the unique pseudo-expectation property,
then so does $(C^*(\D,\D^c),\D)$, by Proposition \ref{hereditary from
above}. By Theorem \ref{CE bijection}, there exists a unique
conditional expectation $C^*(\D,\D^c) \to \D$, therefore a unique
conditional expectation $\D^c \to Z(\D)$. By Theorem \ref{abelian
W^*}, $Z(\D)' \cap \D^c = Z(\D)$. But
\[ Z(\D)' \cap \D^c = Z(\D)' \cap \D' \cap \M = \D' \cap \M = \D^c.
\]
\end{proof}

We now turn to our main purpose in
this section---characterizing the unique pseudo-expectation property
for various classes of $W^*$-inclusions (Theorems~\ref{BH-inclusions}
and \ref{W^*-inclusions}).

The statement of Theorem \ref{W^*-inclusions} involves the tracial
ultrapower construction, which we recall for the reader. Let $\M$ be a
$II_1$ factor with trace $\tau$, and let $\omega \in \beta\bbN
\backslash \bbN$ be a free ultrafilter. The \emph{tracial ultrapower}
of $\M$ with respect to $\omega$ is defined to be $\M^\omega =
\ell^\infty(\M)/\I_\omega$, where
\[ \I_\omega = \{(x_n) \in \ell^\infty(\M): \lim_\omega \tau(x_n^*x_n)
= 0\} \ideal \ell^\infty(\M).
\] It can be shown that $\M^\omega$ itself is a $II_1$ factor with
trace
\[ \tau_\omega((x_n) + \I_\omega) = \lim_\omega \tau(x_n).
\] The map $\M \to \M^\omega: x \mapsto (x) + \I_\omega$ is an
embedding. If $\D \subseteq \M$ is a MASA, then $\D^\omega =
(\ell^\infty(\D) + \I_\omega)/\I_\omega \subseteq \M^\omega$ is a
MASA. See \cite[Appendix A]{SinclairSmith2008} for more details.

The proofs of Theorems~\ref{BH-inclusions} and \ref{W^*-inclusions}
require some standard facts about conditional expectations, which we
collect into a proposition for the reader's convenience.

\begin{proposition} \label{CE facts} \hfill
\begin{enumerate}
\item Let $I$ be an index set and for $i \in I$, let $\M_i \subseteq B(\H_i)$ be a
$W^*$-algebra. Then there exists a bijective correspondence between
families of conditional expectations $\{E_i:B(\H_i) \to \M_i\}_{i \in
I}$ and conditional expectations $\theta:B(\bigoplus_{i \in I} \H_i)
\to \bigoplus_{i \in I} \M_i$. Namely
    \[ \theta(\begin{bmatrix} x_{ij} \end{bmatrix}) = \oplus_{i \in I}
E_i(x_{ii}).
    \] We have that $\theta$ is normal (resp. faithful) iff every
$E_i$ is normal (resp. faithful).
\item For $i = 1, 2$, let $(\M_i,\D_i)$ be a $W^*$-inclusion and
$E_i:\M_i \to \D_i$ be a conditional expectation. Then there exists a
conditional expectation $E:\M_1 \overline{\otimes} \M_2 \to \D_1
\overline{\otimes} \D_2$ such that $E(x_1 \otimes x_2) = E_1(x_1)
\otimes E_2(x_2)$ for all $x_1 \in \M_1$, $x_2 \in \M_2$. If $E_1$ and
$E_2$ are normal, then there exists a unique normal conditional
expectation $E$ as above. \cite[Thm. 4]{Tomiyama1969}
\item Let $(\M,\D)$ be a $W^*$-inclusion. Then there exists a
bijective correspondence between conditional expectations $E:\M \to
\D$ and conditional expectations $\theta:B(\K) \overline{\otimes} \M
\to B(\K) \overline{\otimes} \D$. Namely
    \[ \theta(\begin{bmatrix} x_{ij} \end{bmatrix}) = \begin{bmatrix}
E(x_{ij}) \end{bmatrix}.
    \] We have that $E$ is normal (resp. faithful) iff $\theta$ is
normal (resp. faithful).
\end{enumerate}
\end{proposition}

Now we come to the main results of this section.  The first
characterizes the unique pseudo-expectation property for
$W^*$-inclusions of the form $(\B(\H),\D)$, and the second characterizes
the unique pseudo-expectation property for $W^*$-inclusions
$(\M,\D)$ when $\M_*$ is separable and $\D$ is abelian. In the
latter result, the separability hypothesis cannot be removed.

\begin{theorem}\label{BH-inclusions}\hfill
\begin{enumerate}
\item Let $\A \subseteq B(\H)$ be an abelian $W^*$-algebra. Then
$(B(\H),\A)$ has the unique pseudo-expectation property iff $\A$ is an
atomic MASA. The unique pseudo-expectation is a normal faithful
conditional expectation.
\item Generalizing (i), let $\M \subseteq B(\H)$ be a
$W^*$-algebra. Then $(B(\H),\M)$ has the unique pseudo-expectation
property iff $\M'$ is abelian and atomic. The unique
pseudo-expectation is a normal faithful conditional expectation. In
particular, $\M$ is type $I$ (and injective).
\end{enumerate}
\end{theorem}
\begin{proof}
(i) By Theorem \ref{abelian W^*}, we may assume that $\A
\subseteq B(\H)$ is a MASA. We have the unitary equivalence
\[ \A = \A_{\text{atomic}} \oplus \A_{\text{diffuse}},
\] where $\A_{\text{atomic}}$ is spatially isomorphic to
$\ell^\infty(\kappa)$ acting on $\ell^2(\kappa)$ for some index set
$\kappa$, and $\A_{\text{diffuse}}$ is spatially isomorphic to
$\oplus_{i \in I} L^\infty([0,1]^{\alpha_i})$ acting on
$\oplus_{i \in I} L^2([0,1]^{\alpha_i})$ for some index set $I$ and
cardinals $\alpha_i$, $i \in I$ (see~\cite{MaharamOnHoMeAl}). It is
easy to see that for $\kappa$ nonempty, there exists a unique
conditional expectation $B(\ell^2(\kappa)) \to \ell^\infty(\kappa)$,
which is normal and faithful. On the other hand, for any nonzero
cardinal $\alpha$, there are multiple conditional expectations
$B(L^2([0,1]^\alpha)) \to L^\infty([0,1]^\alpha)$. Indeed, this is
well-known when $\alpha = 1$ (Example \ref{diffuse MASA}), and follows
from Proposition \ref{CE facts} (ii) and the unitary equivalence
\[ L^\infty([0,1]^\alpha) = L^\infty([0,1]) \overline{\otimes}
L^\infty([0,1]^\beta) \subseteq B(L^2([0,1])) \overline{\otimes}
B(L^2([0,1]^\beta))
\] when $\alpha > 1$ (here $\beta = \alpha - 1$ if $\alpha$ is finite,
and $\beta = \alpha$ if $\alpha$ is infinite). The result now follows
from Proposition \ref{CE facts} (i).

(ii) By Theorem \ref{W^* ARC}, we may assume that $\M' = Z(\M)$, so
that
\[ \H = \bigoplus_m \ell_m^2 \otimes \H_m, ~ \M' = \bigoplus_m I_m
\overline{\otimes} \A_m, \text{ and } \M = \bigoplus_m B(\ell_m^2)
\overline{\otimes} \A_m,
\] where $\A_m \subseteq B(\H_m)$ is a MASA for each $m$. In particular, $\M$ is injective, so that pseudo-expectations for $(B(\H),\M)$ are actually conditional expectations. By Proposition \ref{CE facts} (i) and
(iii), there exists a unique conditional expectation $B(\H) \to \M$
iff there exist unique conditional expectations $B(\H_m) \to \A_m$ for
each $m$, iff $\A_m$ is atomic for each $m$ (by part (i) above). The
result now follows.
\end{proof}

\begin{theorem} \label{W^*-inclusions}\hfill
\begin{enumerate}
\item Let $(\M,\D)$ be a $W^*$-inclusion, with $\M_*$ separable
and $\D$ abelian. Then $(\M,\D)$ has the unique pseudo-expectation
property iff $\M$ is type $I$, $\D$ is a MASA, and there exists a
family $\{p_t\}$ of abelian
projections for $\M$ such that $\{p_t\} \subseteq \D$ and $\sum_t p_t = 1$. The
unique pseudo-expectation is a normal faithful conditional
expectation.
\item Let $(\M,\D)$ be a $W^*$-inclusion, with $\M$ a $II_1$
factor and $\D$ a singular MASA. If $\omega \in \beta\bbN \backslash
\bbN$, then $(\M^\omega,\D^\omega)$ has the faithful unique pseudo-expectation
property. The unique pseudo-expectation is a normal,
trace-preserving conditional expectation.
\end{enumerate}
\end{theorem}

\begin{proof} (i) ($\Rightarrow$) Suppose there exists a unique
expectation $E:\M \to \D$. Then $\D$ is a MASA, by Theorem
\ref{abelian W^*}. Since $\M_*$ is separable, $\D$ is
singly-generated. Then $E$ is normal and faithful, by
\cite[Cor. 3.3]{AkemannSherman2012}. It follows that $\M$ is type $I$,
by \cite[Thm. 3.3]{PopaVaes2015}. Thus there exist abelian projections $\{p_t\}$
for $\M$ such that $\{p_t\} \subseteq \D$ and $\sum_t p_t = 1$, by
\cite[Thm. 4.1]{AkemannSherman2012}.

($\Leftarrow$) Conversely, if $\M$ is type $I$, $\D$ is a MASA, and
there exist abelian projections $\{p_t\}$ for $\M$ such that $\{p_t\} \subseteq \D$
and $\sum_t p_t = 1$, then there exists a unique conditional expectation
$E:\M \to \D$, by \cite[Thm. 4.1]{AkemannSherman2012}.

(ii) By \cite[Thm 0.1]{Popa2014}, $(\M^\omega,\D^\omega)$ has the
unique extension property (\textsf{UEP}). By Example \ref{UEP MASA},
$(\M^\omega,\D^\omega)$ has the unique pseudo-expectation property,
and the unique pseudo-expectation is a conditional expectation,
necessarily normal, faithful, and trace-preserving.
\end{proof}

\begin{remark} Contrasting statements (i) and (ii) of Theorem
\ref{W^*-inclusions} above, we see that separability plays a role in
the unique pseudo-expectation property.
\end{remark}

\section{Applications}

In this section we show that the faithful unique pseudo-expectation
property can substantially simplify \emph{$C^*$-envelope}
calculations, and we relate the faithful unique pseudo-expectation
property to \emph{norming} in the sense of Pop, Sinclair, and Smith.

\subsection{$C^*$-Envelopes}

Let $\C$ be a unital $C^*$-algebra and $\X \subseteq \C$ be a unital
operator space such that $C^*(\X) = \C$. There exists a unique
maximal closed two-sided ideal $\J \ideal \C$ such that quotient map
$q:\C \to \C/\J$ is completely isometric on $\X$
\cite{Hamana1979a}. Then $C_e^*(\X) = \C/\J$ is the
\emph{$C^*$-envelope} of $\X$, the (essentially) unique minimal
$C^*$-algebra generated by a completely isometric copy of $\X$. In
general, determining $C_e^*(\X)$ can be quite challenging.   However, if
$\D \subseteq \X \subseteq \C$, and $(\C,\D)$ has the faithful
unique pseudo-expectation property, then determining $C_e^*(\X)$ is
not hard at all.

\begin{theorem} \label{envelope} Let $(\C,\D)$ be a $C^*$-inclusion
with the faithful unique pseudo-expectation property (more generally,
such that every pseudo-expectation is faithful). If $\D \subseteq \X
\subseteq \C$ is an operator space, then $C_e^*(\X) = C^*(\X)$. That
is, the $C^*$-envelope equals the generated $C^*$-algebra.
\end{theorem}

\begin{proof} By the previous discussion, $C_e^*(\X) = C^*(\X)/\J$,
where $\J \ideal C^*(\X)$ is the unique maximal closed two-sided ideal
such that $q:C^*(\X) \to C^*(\X)/\J$ is completely isometric on
$\X$. Since $\D \subseteq \X$, $\J$ must be $\D$-disjoint. But then
$\J = 0$, since $(\C,\D)$ is hereditarily essential, by Theorem
\ref{faithful characterization}.
\end{proof}

We say that a $C^*$-inclusion $(\C,\D)$ is \emph{$C^*$-envelope
determining} if $C_e^*(\X) = C^*(\X)$ for every operator space $\D
\subseteq \X \subseteq \C$. With this terminology, Theorem
\ref{envelope} becomes the implication
\[ \text{every pseudo-expectation faithful} \implies
\text{$C^*$-envelope determining}.
\] The converse is false.

\begin{example} Let $\C = M_{2 \times 2}(\bbC)$ and $\D = \bbC
I$. Then $(\C,\D)$ is $C^*$-envelope determining, but admits multiple
pseudo-expectations, some of which are not faithful.
\end{example}

\begin{proof} Let $\bbC \subseteq \X \subseteq M_{2 \times 2}(\bbC)$
be an operator space. Then $\dim(C^*(\X)) \in \{1, 2, 4\}$. If
$\dim(C^*(\X)) \in \{1, 2\}$, then $C^*(\X) = \X$, which implies
$C_e^*(\X) = \X$. Otherwise, if $\dim(C^*(\X)) = 4$, then $C^*(\X) =
M_{2 \times 2}(\bbC)$, which implies $C_e^*(\X) = C^*(\X)$, since
$M_{2 \times 2}(\bbC)$ is simple.
\end{proof}

\subsection{Norming}

According to Pitts' Theorem \ref{Pitts}, if $(\C,\D)$ is a regular MASA
inclusion with the faithful unique pseudo-expectation property, then
$\D$ \emph{norms} $\C$ in the sense of Pop, Sinclair, and Smith
\cite{PopSinclairSmith2000}. In this section we investigate the
relationship between the faithful unique pseudo-expectation property
and norming, for arbitrary $C^*$-inclusions. We show that the faithful
unique pseudo-expectation is conducive to norming (Theorem
\ref{conducive norming}), but does not imply it (Example \ref{not
norming}).

We begin by recalling the definition of norming, and proving some
general norming results which we will need later. Some of these
results may be of independent interest.

\begin{definition} We say that an inclusion $(\C,\D)$ is
\textbf{norming} if for any $X \in M_{d \times d}(\C)$, we have that
\[ \|X\| = \sup\{\|RXC\|: R \in \Ball(M_{1 \times d}(\D)), C \in
\Ball(M_{d \times 1}(\D))\}.
\]
\end{definition}

\begin{proposition} \label{exhaustion norming} Let $(\M,\D)$ be a
$W^*$-inclusion and $\{p_t\} \subseteq \D$ be an increasing net of
projections such that $\sup_t p_t = 1$. If $(p_t\M p_t,p_t\D p_t)$ is
norming for all $t$, then $(\M,\D)$ is norming.
\end{proposition}

\begin{proof} Let $\H$ be the Hilbert space on which $\M$ acts. Fix $X
\in M_{d \times d}(\M)$ and $\epsilon > 0$. There exist $\xi, \eta \in
\Ball(\H^d)$ such that
\[ |\langle X\xi, \eta \rangle| > \|X\| - \epsilon.
\] Since $\sup_t p_t = 1$, there exists $t$ such that
\[ |\langle X(I_d \otimes p_t)\xi, (I_d \otimes p_t)\eta \rangle| >
|\langle X\xi, \eta \rangle| - \epsilon.
\] Set $\tilde{X} = (I_d \otimes p_t)X(I_d \otimes p_t) \in M_{d
\times d}(p_t\M p_t)$. Since $p_t\D p_t$ norms $p_t\M p_t$, there
exist $R \in \Ball(M_{1 \times d}(p_t\D p_t))$, $C \in \Ball(M_{d
\times 1}(p_t\D p_t))$ such that
\[ \|R\tilde{X}C\| > \|\tilde{X}\| - \epsilon.
\] Then $R \in \Ball(M_{1 \times d}(\D))$, $C \in \Ball(M_{d \times
1}(\D))$, and
\begin{eqnarray*} \|RXC\| &=& \|R(I_d \otimes p_t)X(I_d \otimes
p_t)C\| = \|R\tilde{X}C\|\\ &>& \|\tilde{X}\| - \epsilon \geq |\langle
\tilde{X}\xi, \eta \rangle| - \epsilon\\ &=& |\langle X(I_d \otimes
p_t)\xi, (I_d \otimes p_t)\eta \rangle| - \epsilon > |\langle X\xi,
\eta \rangle| - 2\epsilon\\ &>& \|X\| - 3\epsilon.
\end{eqnarray*}
\end{proof}

\begin{corollary} \label{direct sum norming} For $i \in I$, let
$(\M_i,\D_i)$ be a $W^*$-inclusion. If $(\M_i,\D_i)$ is norming for
all $i \in I$, then $(\bigoplus_{i \in I} \M_i,\bigoplus_{i \in I}
\D_i)$ is norming.
\end{corollary}

\begin{lemma} \label{quotient norming} Let $(\C,\D)$ be a
$C^*$-inclusion and $\I \ideal \C$. If $(\D + \I)/\I$ norms $\C/\I$,
then for every $X \in M_{d \times d}(\C)$, there exist $R \in
\Ball(M_{1 \times d}(\D))$ and $C \in \Ball(M_{d \times 1}(\D))$ such
that
\[ \|RXC\| > \|X + M_{d \times d}(\I)\| - \epsilon.
\]
\end{lemma}

\begin{proof} Let $\pi:\C \to \C/\I$ be the quotient map. Fix $X \in
M_{d \times d}(\C)$ and $\epsilon > 0$. By assumption, there exist $R
\in M_{1 \times d}(\D)$ and $C \in M_{d \times 1}(\D)$ such that
$\|\pi_{1 \times d}(R)\| < 1$, $\|\pi_{d \times 1}(C)\| < 1$, and
\[ \|\pi_{1 \times d}(R)\pi_{d \times d}(X)\pi_{d \times 1}(C)\| >
\|\pi_{d \times d}(X)\| - \epsilon.
\] Now the map
\[ \D/(\D \cap \I) \to (\D + \I)/\I: d + \D \cap \I \mapsto d + \I
\] is a unital $*$-isomorphism, in particular a complete
isometry. Thus
\[ \|\pi_{1 \times d}(R)\| = \|R + M_{1 \times d}(\I)\| = \|R + M_{1
\times d}(\D \cap \I)\|.
\] It follows that there exists $\tilde{R} \in M_{1 \times d}(\D)$
such that $\|\tilde{R}\| < 1$ and $\pi_{1 \times d}(\tilde{R}) =
\pi_{1 \times d}(R)$. Likewise there exists $\tilde{C} \in M_{d \times
1}(\D)$ such that $\|\tilde{C}\| < 1$ and $\pi_{d \times 1}(\tilde{C})
= \pi_{d \times 1}(C)$. Then
\[ \|\tilde{R}X\tilde{C}\| \geq \|\pi_{1 \times d}(\tilde{R})\pi_{d
\times d}(X)\pi_{d \times 1}(\tilde{C})\| > \|\pi_{d \times d}(X)\| -
\epsilon.
\]
\end{proof}

\begin{proposition} \label{ultrapower norming} Let $\M$ be a $II_1$
factor, $\D \subseteq \M$ be a MASA, and $\omega \in \beta\bbN
\backslash \bbN$. If $(\M^\omega,\D^\omega)$ is norming, then so is
$(\M,\D)$.
\end{proposition}

\begin{proof} By assumption, $\D^\omega = (\ell^\infty(\D) +
\I_\omega)/\I_\omega$ norms $\M^\omega =
\ell^\infty(\M)/\I_\omega$. Let $X \in M_{d \times d}(\M)$ and
$\epsilon > 0$. Then $(X) \in M_{d \times d}(\ell^\infty(\M)) =
\ell^\infty(M_{d \times d}(\M))$. By Lemma \ref{quotient norming},
there exist $(R_n) \in M_{1 \times d}(\ell^\infty(\D)) =
\ell^\infty(M_{1 \times d}(\D))$ and $(C_n) \in M_{d \times
1}(\ell^\infty(\D)) = \ell^\infty(M_{d \times 1}(\D))$ such that
$\|(R_n)\| < 1$, $\|(C_n)\| < 1$, and
\[ \|(R_n)(X)(C_n)\| > \|(X) + M_{d \times d}(\I_\omega)\| - \epsilon.
\] Thus
\[ \sup_n \|R_n\| < 1, ~ \sup_n \|C_n\| < 1, \text{ and } \sup_n
\|R_nXC_n\| > \|X\| - \epsilon.
\]
\end{proof}

Now we list some classes of \cstar-inclusions for which (\textsf{f!PsE}) $\implies$ (\textsf{Norming}).

\begin{theorem} \label{conducive norming} For the following classes of
$C^*$-inclusions, the faithful unique pseudo-expectation property
implies norming:
\begin{enumerate}
\item Regular MASA inclusions $(\C,\D)$. In particular,
transformation group $C^*$-algebras $(C(X) \rtimes_r \Gamma,C(X)) =
(C(X) \rtimes \Gamma,C(X))$.
\item Abelian inclusions $(C(Y),C(X))$.
\item $C^*$-inclusions $(\C,\D)$ with $\D \subseteq \C \subseteq
I(\D)$ (i.e., operator space essential inclusions).
\item $W^*$-inclusions $(B(\H),\M)$.
\item $W^*$-inclusions $(\M,\D)$, with $\M_*$ separable and $\D$
abelian.
\end{enumerate}
\end{theorem}

\begin{proof} (i) This is  Pitts' Theorem \ref{Pitts}.

(ii) Abelian inclusions $(C(Y),C(X))$ are norming, whether or not they
have the faithful unique pseudo-expectation property
\cite[Ex. 2.5]{PopSinclairSmith2000}.

(iii) Let $(\C,\D)$ be an operator space essential inclusion (see the discussion preceding Example \ref{OSE}). For $X
\in M_{d \times d}(\C)$, define
\[ \gamma_d(X) = \sup\{\|RXC\|: R \in \Ball(M_{1 \times d}(\D)), ~ C
\in \Ball(M_{d \times 1}(\D))\}.
\] By \cite[Thm. 2.1]{Magajna1999}, $\gamma = (\gamma_d)_{d=1}^\infty$
is an operator space structure on $\C$ with the following properties:
\begin{itemize}
\item $\gamma_1(x) = \|x\|$, $x \in \C$;
\item $\gamma_d(X) \leq \|X\|$, $X \in M_{d \times d}(\C)$;
\item $\gamma_d(D) = \|D\|$, $D \in M_{d \times d}(\D)$.
\end{itemize} It follows that the identity map $\iota:(\C,\|\cdot\|)
\to (\C,\gamma)$ is a complete contraction which is completely
isometric on $\D$. Since $(\C,\D)$ is operator space essential,
$\iota$ is actually completely isometric on $\C$, so $(\C,\D)$ is norming.

(iv) By Theorem \ref{W^*-inclusions}, $\M'$ is abelian. Let $\M'
\subseteq \A \subseteq B(\H)$ be a MASA. Then $\A = \A' \subseteq \M''
= \M$. Since $\A$ norms $B(\H)$, $\M$ norms $B(\H)$ as well.

(v) By Theorem \ref{W^*-inclusions}, $\M$ is type $I$, $\D$ is a
MASA, and there exist abelian projections $\{p_t\} \subseteq \D$ for $\M$ such that $\sum_t p_t = 1$. Letting $p_F = \sum_{t \in F} p_t$ for every
finite set of indices $F$, we obtain an increasing net $\{p_F\}
\subseteq \D$ of finite projections for $\M$ such that $\sup_F p_F = 1$. By
Proposition \ref{exhaustion norming}, to prove that $(\M,\D)$ is
norming it suffices to prove that $(p_F\M p_F,\D p_F)$ is norming for
each $F$. Thus we may assume that $\M$ is finite type $I$. By
Corollary \ref{direct sum norming}, we may further assume that $\M$ is
type $I_n$ for some $n \in \bbN$. There is a (non-spatial)
$*$-isomorphism $\M = M_{n \times n}(\A) \subseteq B(\H^n)$, where $\A
\subseteq B(\H)$ is a MASA and $\H$ is separable. By
\cite[Thm. 3.19]{Kadison1984}, there exists a unitary $u \in \M$ such
that $u\D u^* = \ell_n^\infty(\A)$. It follows that $\D \subseteq
B(\H^n)$ is a MASA, which implies that $\D$ norms $B(\H^n)$
\cite[Thm. 2.7]{PopSinclairSmith2000}. Thus $\D$ norms $\M$.
\end{proof}

\begin{example} \label{not norming} There exists a $II_1$ factor $\M$
and a singular MASA $\D \subseteq \M$ such that $(\M,\D)$ has the
faithful unique pseudo-expectation property, but $\D$ does not norm
$\M$. Of course $\M_*$ is non-separable.
\end{example}

\begin{proof} Let $\bbF_2$ be the free group on two generators $u$ and
$v$, and let $\D = W^*(u) \subseteq W^*(\bbF_2) = \M$. Then $\M$ is a
$II_1$ factor, $\D$ is a singular MASA, and $\D$ does not norm $\M$
\cite[Thm. 5.3]{PopSinclairSmith2000}. Now let $\omega \in \beta\bbN
\backslash \bbN$. Then $\M^\omega$ is a $II_1$ factor, $\D^\omega$ is
a singular MASA, and $(\M^\omega,\D^\omega)$ has the faithful unique
pseudo-expectation property, by Theorem \ref{W^*-inclusions}. But
$\D^\omega$ does not norm $\M^\omega$, by Proposition \ref{ultrapower
norming}.
\end{proof}

\section{Conclusion}

We conclude this paper with a list of questions and partial progress toward some answers.

\subsection{Questions}

\begin{enumerate}
\item[Q1] Is the condition (\textsf{Reg}) hereditary from above? That is, if $(\C,\D)$ is a regular inclusion and $\D \subseteq \C_0 \subseteq \C$ is a $C^*$-algebra, is $(\C_0,\D)$ a regular inclusion? (We expect that the answer is ``no''.)
\item[Q2] Is the condition (\textsf{UEP}) hereditary from below? That is, if $(\C,\D)$ has the unique extension property and $\D \subseteq \D_0 \subseteq \C$ is a $C^*$-algebra, does $(\C,\D_0)$ have the unique extension property? (Again, we expect that the answer is ``no''.)
\item[Q3] Is there a $C^*$-inclusion $(\C,\D)$ with a unique conditional expectation, but multiple pseudo-expectations?
\item[Q4] If $(\C,\D)$ has the unique extension property (\textsf{UEP}), does $(\C,\D)$ have the unique pseudo-expectation property? (By Example \ref{UEP MASA}, the answer is ``yes'' if $\D$ is abelian.)
\item[Q5] If $(\C,\D)$ has the unique pseudo-expectation property, is $\D^c$ abelian? What if $\D$ is abelian? (By Corollary \ref{!fPse implies ARC}, the answer is ``yes'' if $(\C,\D)$ has the faithful unique pseudo-expectation property.)
\item[Q6] If every pseudo-expectation is faithful, is there a unique pseudo-expectation? Equivalently, by Theorem \ref{faithful characterization}, does $(\C,\D)$ have the faithful unique pseudo-expectation property iff $(\C,\D)$ is hereditarily essential?
\item[Q7] Let $\Gamma$ be a discrete group acting on a compact Hausdorff space $X$ by homeomorphisms. Find a condition on the action equivalent to $(C(X) \rtimes_r \Gamma,C(X))$ having the unique pseudo-expectation property. (By Theorem \ref{crossed}, if $C(X) \rtimes_r \Gamma = C(X) \rtimes \Gamma$, then $(C(X) \rtimes_r \Gamma,C(X))$ has the unique pseudo-expectation property iff the action of $\Gamma$ on $X$ is topologically free.)
\item[Q8] Is the $C^*$-inclusion $(B(\ell^2)/K(\ell^2),\ell^\infty/c_0)$ norming?
\item[Q9] Is there a condition on a $C^*$-inclusion $(\C,\D)$ which together with the faithful unique pseudo-expectation property implies norming? In particular, is the separability of $\C$ such a condition?
\item[Q10] Is there a condition on a $W^*$-inclusion $(\M,\D)$ which together with the faithful unique pseudo-expectation property implies norming? In particular, is the separability of $\M_*$ such a condition? (By Theorem \ref{conducive norming}, the answer is ``yes'' if $\D$ is abelian.)
\end{enumerate}

\subsection{Progress on Questions 5 and 6}
\subsubsection*{\textsc{Question 5:}}
We are able to show that if $(\C,\D)$ is a
$C^*$-inclusion with $\D$ abelian, such that there exists a unique
pseudo-expectation $\Phi$, then $\Phi$ is multiplicative on $\D^c$
(Proposition \ref{relative commutant multiplicative}). We regard this
as partial progress toward proving that $\D^c$ is abelian. Indeed, by
Corollary \ref{abelian}, if $\D^c$ is abelian and $\Phi$ is the unique
pseudo-expectation for $(\C,\D)$, then $\Phi$ necessarily is
multiplicative on $\D^c$.

\begin{proposition} \label{relative commutant multiplicative} Let
$(\C,\D)$ be a $C^*$-inclusion, with $\D$ abelian. If $(\C,\D)$ has
unique pseudo-expectation $\Phi \in \PsExp(\C,\D)$, then $\Phi$ is
multiplicative on $\D^c$.
\end{proposition}

\begin{proof} Let
\[ M_\Phi = \{x \in \C: \Phi(x^*x) = \Phi(x)^*\Phi(x), ~ \Phi(xx^*) =
\Phi(x)\Phi(x)^*\}
\] be the \emph{multiplicative domain} of $\Phi$, the largest
$C^*$-subalgebra of $\C$ on which $\Phi$ is multiplicative \cite[Thm.\
3.18]{Paulsen2002}. Suppose $x \in (\D^c)_{sa}$. Then $\C_x =
C^*(\D,x)$ is a unital abelian $C^*$-algebra containing $\D$. By
Proposition \ref{hereditary from above}, $(\C_x,\D)$ has unique
pseudo-expectation $\Phi|_{\C_x}$, which is multiplicative by
Corollary \ref{abelian}. It follows that $x \in M_\Phi$.
\end{proof}

\subsubsection*{\textsc{Question 6:}}
We  show that if $(\M,\D)$ is a
$W^*$-inclusion with $\D$ injective, such that every
pseudo-expectation is faithful, then there exists a unique
pseudo-expectation (Proposition \ref{all faithful implies unique}).

Recall that a bounded linear map $T$ between von Neumann algebras
$\M$ and $\N$ is \textit{singular} if $f\circ T\in (\M_*)^\perp$
whenever $f \in \N_*$.

\begin{lemma} \label{singular cp} Let $(\M,\D)$ be a $W^*$-inclusion
and $\theta:\M \to \D$ be a completely positive $\D$-bimodule map. If
$\theta$ is singular, then for every projection $0 \neq p \in Z(\D)$,
there exists a projection $0 \neq e \in \M$ such that $e \leq p$ and
$\theta(e) = 0$.
\end{lemma}

\begin{proof} Let $\{\phi_i\} \subseteq (\D_*)_+$ be a maximal family
with mutually orthogonal supports $\{s(\phi_i)\} \subseteq \D$. Then
$\sum_i s(\phi_i) = 1$, and so there exists $j$ such that $s(\phi_j)p
\neq 0$. Since $\theta$ is singular, $\phi_j \circ \theta \in
(\M_*)^\perp_+$. Thus there exists a projection $0 \neq e \in \M$ such
that $e \leq s(\phi_j)p$ and $\phi_j(\theta(e)) = 0$
\cite[Thm. III.3.8]{Takesaki1979}. It follows that
$s(\phi_j)\theta(e)s(\phi_j) = 0$, which implies
\[ \theta(e) = \theta(s(\phi_j)es(\phi_j)) =
s(\phi_j)\theta(e)s(\phi_j) = 0.
\]
\end{proof}

\begin{lemma} \label{bimodule factorization} Let $(\M,\D)$ be a
$W^*$-inclusion and $\theta:\M \to \D$ be a completely positive
$\D$-bimodule map. If $\Exp(\M,\D) \neq \emptyset$, then there exists
$E \in \Exp(\M,\D)$ such that $\theta(x) = \theta(1)E(x)$, $x \in
\M$. Furthermore, if $0 \leq x \leq s(\theta(1))$ and $\theta(x) = 0$,
then $E(x) = 0$.
\end{lemma}

\begin{proof} Fix $E_0 \in \Exp(\M,\D)$. Let $p = s(\theta(1)) \in
Z(\D)$ and define
\[ E(x) = \lim_{k \to \infty} (\theta(1) + 1/k)^{-1}\theta(x) +
p^\perp E_0(x), ~ x \in \M,
\] where the limit exists in the strong operator topology (see
\cite[Lemma 5.1.6]{EffrosRuan2000}). Then $E \in \Exp(\M,\D)$ and
$\theta(x) = \theta(1)E(x)$, $x \in \M$. If $0 \leq x \leq p$ and
$\theta(x) = 0$, then $0 \leq E(x) \leq p$.  But $E(x)=p^\perp
E_0(x)$,
which implies $E(x) = 0$.
\end{proof}

\begin{lemma} \label{all faithful implies all normal} Let $(\M,\D)$ be
a $W^*$-inclusion. If every conditional expectation $\M \to \D$ is
faithful, then every conditional expectation $\M \to \D$ is normal.
\end{lemma}

\begin{proof} Let $E:\M \to \D$ be a conditional expectation. By
\cite{Tomiyama1959}, there exist completely positive $\D$-bimodule
maps $\theta_n, \theta_s:\M \to \D$ such that $\theta_n$ is normal,
$\theta_s$ is singular, and $E = \theta_n + \theta_s$. Assume that
$\theta_s \neq 0$. By Lemma \ref{singular cp}, there exists a
projection $0 \neq e \in \M$ such that $e \leq s(\theta_s(1))$ and
$\theta_s(e) = 0$. By Lemma \ref{bimodule factorization}, there exists
a conditional expectation $E_s:\M \to \D$ such that $E_s(e) = 0$, a
contradiction. Thus $E = \theta_n$ is normal.
\end{proof}

\begin{proposition} \label{all faithful implies unique} Let $(\M,\D)$
be a $W^*$-inclusion, with $\D$ injective. If every pseudo-expectation
$\Phi \in \PsExp(\M,\D)$ is faithful, then $(\M,\D)$ has the unique
pseudo-expectation property.  In this situation, the unique
pseudo-expectation is a conditional expectation which is faithful and normal.
\end{proposition}

\begin{proof} Since $\D$ is injective, $\PsExp(\M,\D) =
\Exp(\M,\D)$. By Lemma \ref{all faithful implies all normal}, every
conditional expectation $\M \to \D$ is faithful and normal. By
\cite[Thm. 5.3]{CombesDelaroche1975}, the map $E \mapsto E|_{\D^c}$ is
a bijection between the faithful normal conditional expectations $\M
\to \D$ and the faithful normal conditional expectations $\D^c \to
Z(\D)$. Thus to show that there exists a unique conditional
expectation $\M \to \D$, it suffices to show that there exists a
unique faithful normal conditional expectation $\D^c \to Z(\D)$. In
fact, we will show that $\D^c = Z(\D)$.

By Theorem \ref{not ARC}, $\D^c$ is abelian. Since every conditional
expectation $\M \to \D$ is faithful, every conditional expectation
$C^*(\D,\D^c) \to \D$ is faithful, which implies every conditional
expectation $\D^c \to Z(\D)$ is faithful, by Theorem \ref{CE
bijection}. In particular, every multiplicative conditional
expectation $\D^c \to Z(\D)$ is faithful, which implies $\D^c =
Z(\D)$.
\end{proof}

\end{document}